\documentclass[11pt, openright, twosides]{article}

\usepackage{amssymb,latexsym}
\usepackage{amsmath,amscd}
\usepackage[all]{xy}
\usepackage{fancyhdr}
\usepackage{indentfirst}
\usepackage{graphicx}
\usepackage{newlfont}
\usepackage{amsfonts}
\usepackage{amsthm}
\usepackage{mathrsfs}
\usepackage{psfrag}
\usepackage[all]{xy}

\renewcommand{\i}       {\underline{i}}
\renewcommand{\j}       {\underline{j}}
\renewcommand{\k}       {\underline{k}}
\newcommand{\eps}       {\varepsilon}

\newcommand{\la}        {\lambda}
\newcommand{\R}         {\mathbb{R}}
\renewcommand{\H}       {\mathbb{H}}
\newcommand{\C}         {\mathbb{C}}

\newcommand{\E}         {\mathbb{E}}
\renewcommand{\P}       {\mathbb{P}}
\newcommand{\Z}         {\mathbb{Z}}

\renewcommand{\epsilon} {\varepsilon}

\newtheorem{lemma} {Lemma} [section]
\newtheorem{prop} [lemma] {Proposition}

\newtheorem{theo} [lemma] {Theorem}

\newtheorem{defin}[lemma] {Definition}
\newtheorem{rmk}{Remark}

\begin{document}

\title{The cobordism class of the moduli space of polygons in $\R^3$}
\author{Alessia Mandini
\thanks{Partially supported by the Funda\c{c}\~{a}o para a Ci\^{e}ncia e a Tecnologia through the Program POCI 2010/FEDER.}}

\maketitle

{\bf Abstract.} For any vector $r=(r_1,\ldots, r_n)$, let $M_r$ denote the moduli space
(under rigid motions) of polygons in $\R^3$ with $n$-sides whose
lengths are $r_1,\ldots,r_n$.  We give an explicit characterization
of the oriented $S^1$-cobordism class of $M_r$ which depends uniquely on the
length vector $r$.

\section{Introduction}
The study of the geometry of moduli spaces of polygons with fixed side lengths $r_1, \ldots, r_n$ in the Euclidean space has raised, since the 1990's, a remarkable interest in symplectic geometry.  These moduli spaces have a very rich structure; they can be described (in two possible ways) as symplectic quotients: see for example \cite{km} where  Kapovich and Millson show that these spaces are complex-analytic spaces and they define and study the Hamiltonian flows on $M_r$ obtained by bending polygons along diagonals. Another description of $M_r$ as a symplectic reduction is given by Hausmann and Knutson \cite{hk1}, who also give a useful geometric interpretation of the bending action. 

 Let  $\mathcal{S}_r = \prod_{j=1}^n S_{r_j}^2$ be the product of $n$ spheres of radii $r_1, \ldots ,r_n $ respectively; $\mathcal{S}_r$ is a symplectic manifold and a Hamiltonian $SO(3)$-space with associated moment map 
\begin{displaymath}
\begin{array}{rcl}
\mu : \mathcal{S}_r& \rightarrow &\textrm{Lie}(SO(3))^* \simeq \R^3\\
\vec{e}= (e_1, \ldots, e_n) & \mapsto & e_1 + \ldots + e_n.\\
\end{array}
\end{displaymath}
For a (suitably chosen) length vector $r=(r_1, \ldots, r_n)\in\R^n_+$ the symplectic quotient $\mathcal{S}_r / \! \! /SO(3)$ at the 0-level set is a smooth manifold, and it is defined to be the moduli space $M_r$ (Kapovich--Millson \cite{km}). Note that the condition $\mu(\vec{e})=0$ is the closing condition for a polygon with edge vectors $e_1, \ldots, e_n$ starting at an arbitrary base-point. Thus $M_r$ can be identified with the set of polygons in $\R^3$, with $n$ sides of lengths $r_1,\ldots, r_n$, modulo rigid motions.

$M_r$ can also be described as the symplectic reduction for the natural action of the torus $U_1^n$, of diagonal matrices in the unitary group $U_n$, on the complex Grassmannian of 2-planes $Gr_2(\mathbb{C}^n)$ (Hausmann--Knutson \cite{hk1}); the moment map $\mu_{_{U^n_1}}: Gr_{2,n} \rightarrow \R^n$ associated to this Hamiltonian action maps the plane $ \langle a,b \rangle$  generated by the vectors $a,b \in \C^n$ into $\mu_{_{U^n_1}}(\langle a,b \rangle )= (|a_1|^2 + |b_1|^2, \ldots , |a_n|^2 + |b_n|^2).$ Then $M_r$ is the topological quotient $\mu_{_{U^n_1}}^{-1}(r)/U^n_1.$ 

The main result of this paper (Theorem \ref{t1}) is an explicit characterization of the oriented $S^1$-cobordism
class of $M_r$ which depends uniquely upon a special family of index sets defined as follows:

\begin{defin}\label{df}
For each index set $I \subset \{1, \ldots, n-2 \}$ let $\eps_i=1 $ if $i \in I$ and $\eps_i=-1 $ if $i \in I^c:= \{1, \ldots , n-2\} \setminus I$. An index set $I$ is said to be \emph{$r$-admissible} (or \emph{triangular}, as in \cite{AG}) if and only if the following inequalities hold:
\begin{equation}\label{tr}
\left\{ \begin{array}{l}
\sum \eps_i r_i + r_{n-1} -r_n >0\\
\sum \eps_i r_i - r_{n-1} +r_n >0\\
-\sum \eps_i r_i + r_{n-1} +r_n >0.\\
\end{array} \right.
\end{equation}

\end{defin}

We denote by $\mathcal{I}_r$ the set of all $r$-admissible $I.$  Moreover, if $M$ is a smooth oriented manifold,  we will denote by $-M$ the same manifold with opposite orientation and by $\amalg$ the disjoint union (or topological sum) of smooth manifolds.

\begin{theo}\label{t1}
Let  $r \in \R_+^n$ be such that $M_r $ is a smooth manifold and there exists $i,j \in \{1, \ldots,n \}$ such that $r_i \neq r_j.$ Then the following oriented $S^1$-cobordism holds $$ M_r \sim \coprod_{\substack{I \in \mathcal{I}_r \\ \ell= |I| }} (-1)^{n- \ell} \C\P^{n-3},$$ 
where $M_r$ carries the bending action associated to $r_i $ and $r_j$ and the projective spaces $\C\P^{n-3}$ carry the standard projective $S^1$-action.
In particular $M_r \sim 0$ if $n$ is even.
\end{theo}

The bending action has been introduced by Kapovich-Millson \cite{km} and is described in detail in Section \ref{bending}. The geometrical idea underlying its construction the following: let $P$ be a $n-$gon and $\mu_k$ its $k$-th diagonal, i.e. $\mu_{k}= e_1+\cdots+e_{k+1}.$ Consider the surface $S$ bounded by $P; $ $S$ is the union of the triangles $\Delta_1, \ldots, \Delta_n$ where $\Delta_j $ has edges $\mu_{j-1}, e_{j+1}, \mu_j.$ Each (nonzero) diagonal breaks $S$ in two pieces, $S'$ and $S'',$ $S'$ being the union of $\Delta_1, \ldots, \Delta_k$ and $S''$ the union of the remaining ones. The bending action along the $k$-th diagonal is the $S^1$-action which bends $S'$ along $\mu_k$ and let $S''$ fixed.

The bending along a diagonal $\mu_k$ defines an $S^1$-action on the whole $M_r$ when $\mu_k(P)\neq 0$ for all $P \in M_r,$ see Section \ref{bending} The proof of Theorem \ref{t1} takes in consideration the bending along the last diagonal $\mu_{(n-3)},$ which has never length 0 if $r_{n-1} \neq r_n.$ Since $M_r$ is symplectomorphic to $M_{\sigma(r)}$ for any permutation $\sigma$ on the $n$ edges, we can refer to this situation anytime there exists  $i,j \in \{1, \ldots,n \}$ such that $r_i \neq r_j.$ By bending action associated to $r_i$ and $r_j$ we mean the well defined $S^1$-action of bending along $\mu_{(n-3)}$ in $M_{\sigma(r)},$ where $\sigma$ is any permutation that takes $r_i$ and $r_j$ in the last two positions.

Note that if $r_i=r_j$ for all $i, j \in \{1, \ldots , n \}$ (equilateral case) it is not possible to define an $S^1$-action on the whole $M_r$ by bending. Still it is enough to perturb the edges, for example considering $(r_1, \ldots, r_1+\epsilon)$ for arbitrarly small $\epsilon,$ and Theorem \ref{t1} applies. For equilateral $n$-gons, for $n$ odd, Kamiyama \cite{kamiyama} proved a cobordism result using different techniques (note that the equilateral case for even number of edges is always degenerate).

Precisely, he proves that $M_{(1, \ldots, 1)}$ is cobordant to $ (-1)^{m+1} {2m-1 
\choose m} \, \C\P^{2m-2},$ where the number of edges is $n= 2m+1.$ Applying Theorem \ref{t1} to $M_{(1, \ldots, 1, 1+\epsilon)}$ and formally taking the limit for $\epsilon \rightarrow 0,$ one recovers Kamiyama result. In fact, in the equilateral case $\mathcal{I}_r = \{ I \subset \{1, \ldots, n-2 \} \mid |I|= \frac{n-1}{2} \},$ so the orientation of each projective space in Theorem \ref{t1} is $(-1)^{n-\frac{n-1}{2} } = (-1)^{m+1}.$ Moreover, $|\mathcal{I}_r|= {n-2 \choose \frac{n-1}{2} } = {2m-1 \choose m}. $

The proof of Theorem \ref{t1} is based on cobordism results presented by  Ginzburg, Guillemin and Karshon (\cite{cobordism,ggk}). They show that if $M$ is a smooth oriented $2d$-dimensional manifold endowed with a semi-free $S^1$-action, then the $S^1$-oriented cobordism class of $M$ depends only on the fixed point set $(M)^{S^1}.$ Precisely, (finitely many) isolated fixed points contribute to the cobordism class of $M$ with a copy each of the complex projective space $\pm \C\P^d;$ each $k$-codimensional submanifold $X_k$ of fixed points, $k=1, \ldots, N$  contributes to the cobordism class of $M$ with the total space $B_k$ of a fibration $\xymatrix@1{B_k \ar[r]^{\C\P^k} & X_k }$ over $X_k$ with fiber $\C\P^k.$

The $S^1$-action of bending along a proper (i.e. not an edge) diagonal is a quasi-free $S^1$-action on $M_r$ and satisfies the hypothesis of the cobordism theorems just described. The proof of Theorem \ref{t1} is based on the idea, of Migliorini and Reznikov, to analyze the fixed point set of the bending action to calculate the cobordism class of $M_r.$ Precisely, we first show that submanifolds of fixed points do not contribute to the cobordism class of $M_r.$ Then only the isolated fixed points are relevant to determine the class of $M_r,$ and the proof continues with a thorough analysis of the orientation induced from the infinitesimal generator of the bending action on the $\C\P^d$ associated to each fixed point. While writing the paper the author was made aware of \cite{haus} and acknoledges that the computation of the orientation of these projective spaces might equivalently have been done applying results therein.

The layout of the paper is as follows: we first define the moduli space of polygons, both as a symplectic reduction of a product of spheres (cf. Section \ref{inizio}) and of the Grassmannian (cf. Section \ref{PandG}). Also in Section \ref{bending} we recall some important facts on the bending action.
Then we define the Hamiltonian cobordism class that we are studying and state the results on which our proof is based (see Section \ref{cobordism}). Finally, in Section \ref{dim}, we give the proof of our main theorem. In Section \ref{es_cob} we analyze in detail the case $n=5,$ giving an example for each cobordism type. 
\vspace{0.5cm}

{\bf Acknowledgements.} This work has been developed during my Ph.D. studies under the direction of Luca Migliorini, to whom I am extremely grateful for  introducing me to this subject, and for the guidance and support during these years. Also, I would like to thank Leonor Godinho for her comments on an earlier version of this work, Elisa Prato and Gabriele Vezzosi for suggestions. I am grateful to the referee for useful remarks that also led to considerations on complex cobordism, see Remark \ref{complex}. Finally, I thank the Department of Mathematics of the University of Bologna for partial financial support.

\section{The moduli space of polygons}\label{inizio}

An $n$-gon $P$ in the Euclidean space $\E^3$ is determined by its $n$ vertices $v_1, \ldots, v_n$ joined by the oriented edges $e_j= v_{j+1}
- v_j$ ($e_n=v_1-v_n$). A polygon is said to be degenerate if it lies on a line. Let $\mathcal{P}_n$ be the space of all $n$-gons in $\E^3$: two polygons $P=(v_1, \ldots,v_n)$ and $Q=(w_1,\ldots ,w_n)$ are identified if there exists an orientation preserving isometry $g$ of $\E^3$ such that $g(v_i)= w_i$ for $1\leq i\leq n$. For $r= (r_1, \ldots, r_n) \in \R^n_+$, the moduli space $M_r$ is defined to be the space of $n$-gons with fixed side lengths  $r_1, \ldots, r_n$ modulo isometries as above.

The group $\R_+$ acts on $\mathcal{P}_n$ by scaling and this induces an isomorphism $M_r \cong M_{\lambda r}$ for each $\lambda$ in $\R_+.$ Moreover, the group $S_n$ of permutations on $n$ elements  acts on $\mathcal{P}_n$ by permuting the order of the edges, inducing an isomorphism between $M_r$
and $M_{\sigma(r)}$ for each $\sigma \in S_n.$ 

Let $S^2_t$ be the sphere in $\R^3$ of radius $t$ and center the origin.  For $r= (r_1, \ldots, r_n) \in \R^n_+,$ the product $\mathcal{S}_r = \prod_{j=1}^n S_{r_j}^2$ of $n$ copies of spheres is a smooth manifold which can be endowed with a symplectic structure: if $p_j\colon \mathcal{S}_r \rightarrow S_{r_j}^2 $ is the projection on the $j$-th factor and $\omega_j $ is the volume form on the sphere $S_{r_j}^2,$ then the $2$-form $\omega= \sum_{j=1}^n \frac{1}{r_j} p_j^* \omega_j$ on  $\mathcal{S}_r$ is closed and non-degenerate and $(\mathcal{S}_r, \omega) $ is a symplectic manifold.
The group $SO(3)$ acts diagonally on $\mathcal{S}_r$ or, equivalently,
identifying the sphere $S_{r_j}^2$ with a $SO(3)$-coadjoint orbit, the $SO(3)$-action on each sphere is the coadjoint one. The choice of an invariant inner product on the Lie algebra $\mathfrak{so}(3)$ of $SO(3)$ induces an identification $\mathfrak{so}(3)^* \simeq \R^3$ between the dual of $\mathfrak{so}(3)$ and $\R^3.$ So, on each single sphere $S^2_{r_j},$ the moment map associated to the coadjoint action is the inclusion of $S^2_{r_j}$ in $\R^3.$ It follows that the diagonal action of $SO(3)$ on $\mathcal{S}_r$ is still Hamiltonian and, by linearity, it has moment map 
\begin{displaymath}
\begin{array}{rcl}
\mu : \mathcal{S}_r& \rightarrow &\R^3\\
\vec{e}= (e_1, \ldots, e_n) & \mapsto & e_1 + \cdots + e_n.\\
\end{array}
\end{displaymath}
The level set $\mu^{-1}(0):= \tilde{M_r}= \{ \vec{e}=(e_1, \ldots, e_n) \in \mathcal{S}_r : \sum_{i=1}^n e_i =0 \}$ is a submanifold of $\mathcal{S}_r$ because $0$ is a regular value for $\mu.$

Intuitively, if we think at the $e_j$'s as edges of a ``broken line'' $P$ starting at some point in $\R^3$, then the condition $\sum_{i=1}^n
e_i =0$ is the closing condition for $P$ making it a polygon
in $\R^3$. Thus the topological quotient $\tilde{M_r}/SO(3)$ is the moduli space $M_r$ of $n$-gons of fixed side lengths $r$ modulo rigid motions and $M_r$ is realized as the symplectic quotient $\mathcal{S}_r  / \!\! / SO(3)$. 

Kapovich and Millson (\cite{km}) proved that $M_r$ is a smooth manifold if and only if the vector of lengths $r$ does not admit degenerate polygons. Note that  the existence of degenerate polygons in $M_r$ translates into the existence of a partition  $I_1=\{i_1, \ldots, i_s \}$ and $I_2=\{i_{s+1}, \ldots , i_n\}$ of $ \{1, \ldots,n\}$ such that $r_{i_1}+ \cdots +r_{i_s}-r_{i_{s+1}}- \cdots -r_{i_n}=0,$ and thus it is actually a condition on the lengths $r_i.$

If $r \in \R^n_+$ is such that in $M_r$ there exist polygons on a line, then $M_r$ has singularities, which have been studied by Kapovich and Millson in \cite{km}. Precisely, they proved that $M_r$ is a complex analytic space with isolated singularities corresponding to the degenerate $n$-gons in $M_r,$ and these singularities are equivalent' to homogeneous quadratic cones. 

\begin{rmk} Observe that for $\vec{e} \in \tilde{M}_r$ and and $u, v
  \in T_{\vec{e}}\tilde{M}_r$, the formulas $$\! \langle u,v \rangle \! =\! \sum_{j=1}^n \frac{1}{r_j} \langle u_j,v_j \rangle_S, \quad
\omega(u,v)\!= \! \sum_{j=1}^n \langle \frac{e_j}{r_j^2} , u_j \wedge v_j \rangle_S, \quad
J(u)=\!(\ldots,\frac{e_j}{r_j} \wedge u_j,\ldots)$$ 
(where $\langle \, , \rangle_S $ is the standard scalar product in $\R^3$) are
$SO(3)$-invariant, and determine an inner product $\langle \, , \rangle,$ a symplectic form $\omega$, and a complex structure $J$ on $M_r$.
\end{rmk}

\subsection{The Bending Action}\label{bending}
In this section we describe bending flows introduced by Kapovich and Millson in \cite{km}. For each $\vec{e}=(e_1, \ldots,e_n) \in \tilde{M}_r$ let $\mu_k(\vec{e}):= \mu_k $ be its $k$-th diagonal.
The function $f_k(\vec{e})= \frac{1}{2} \| \mu_k\|^2$ is $SO(3)-$invariant, and it will be identified with the function it induces on the quotient space $M_r.$ From now on the construction will depend only formally on the representative of the classes, and $SO(3)$-invariance should be kept in mind. The bending flow around the $k$-th diagonal is the Hamiltonian flow  $\varphi_k^t$ of the Hamiltonian vector field $H_{f_k}$
$$H_{f_k} (e_1, \ldots, e_n )= (\mu_k \wedge e_1, \ldots, \mu_k \wedge e_{k+1}, 0, \ldots, 0) $$
associated to the function $f_k.$

In \cite{km} Kapovich and Millson  prove that $\varphi_k^t$ maps a polygon $P$ of edges $e_1, \ldots, e_n$ into the polygon $\varphi_k^t(P)$  of edges $e_1(t), \ldots, e_n(t),$ where
 \begin{displaymath}\label{system2}
\left\{ \begin{array}{ll}
e_i(t)= \exp(t \,\mathrm{ad}_{\mu_k}) e_i \quad & 1 \leq i \le k+1\\
e_i(t)=e_i , \quad & k+2\le i \le n.\\
\end{array} \right.
\end{displaymath}
From now on we will denote by $\beta_k$ the $S^1$-action just described of bending along the $k$-diagonal. 

Let $\ell_k:M_r \rightarrow \R$ be the function that associates to each polygon $P=\vec{e}$ the length of its $k$-th diagonal, i.e. $\ell_k (P) = \| e_i+ \ldots + e_{k+1}\|, $ then the curve $\varphi_k^t(P)$ is periodic of period $2 \pi/ \ell_k(P)$ if $\ell_k(P)\ne 0,$ otherwise $P$ is a fixed point for $\varphi_k^t$ and the flow  $\varphi_k^t(P)$ has infinite period. It is possible to normalize the flow so that the bending action bends polygons with constant velocity up to excluding the polygons $P$ such that $\ell_k(P)=0.$ Let $M'_r$ be the open subset of $M_r$ consisting of those polygons (called \emph{prodigal}) such that no diagonal $\mu_i$ has zero length; the choice of a system of $n-3$ non intersecting diagonals in $M'_r$ allows one to define an action $\beta$  of a $(n-3)$-dimensional $T^{n-3}$ torus on $M'_r$ by applying progressively the bending actions $\beta_1, \ldots, \beta_{n-3};$ $\beta $ will be called the (toric) bending action. 

Restricting to the dense open subset $M^0_r \subset M'_r$ of polygons such that, for each $i,$ the $i$-th diagonal $\mu_i$ is not collinear to $e_{i+1},$ Kapovich and Millson showed in \cite{km} that this system is completely integrable and  introduced on $M^0_r$ action-angle coordinates. Precisely, the action coordinates are the lengths $\ell_i$ of the diagonals and the angle coordinates are $\theta_i = \pi- \hat{\theta}_i,$ where $\hat{\theta}_i$ is the dihedral angle between $\Delta_i$ and $\Delta_{i+1}.$ (Note that under the hypothesis that no $\mu_i$ is collinear to $e_{i+1}$ none of the $\Delta_i$ is degenerate, thus all the $\theta_i $ are well defined). 

Thus the moment map for the bending action $\beta$ is 
\begin{displaymath}\label{MB} \begin{array}{rcl}
\mu_{_{T^{n-3}}} : M_r & \rightarrow & (\mathfrak{t}^{n-3})^* \simeq \R^{n-3}\\
\vec{e}& \mapsto & (\ell_1(\vec{e}), \ldots,\ell_{n-3}(\vec{e}) ).
\end{array} \end{displaymath}

\begin{rmk}
If $n=4,5,6$ then $M_r $ is toric for generic $r$'s (i.e. for $r$'s such that no degenerate polygons are possible), see \cite{km}.
\end{rmk}

\subsection{Polygon spaces and Grassmannians }\label{PandG}
    
In this section we will briefly overview the description of the moduli space $M_r$ of polygons as the symplectic reduction of the Grassmannian of 2-planes in $\C^n$ by the action of the maximal torus $U_1^n$ of diagonal matrices in $U_n$. This description has been introduced by Hausmann and Knutson in \cite{hk1} and has been used by them (also) to give a nice description of the bending action as the residual torus action coming from the Gel'fand--Cetlin system on $Gr_{2,n}.$ This approach made it possible to study wall-crossing problems and to give an alternative description of the cohomology ring $H^*(M_r)$ (which has been originally computed by Hausmann and Knutson \cite{hk}) by applying the Duistermaat--Heckman Theorem. These results will appear in a further paper \cite{io}. 

The diagonal action of the maximal torus $U^n_1$ on $Gr_{2,n}$ is Hamiltonian with associated moment map $\mu_{_{U^n_1}}: Gr_{2,n} \rightarrow \R^n$ such that, if $\Pi= \langle a,b \rangle$ is the plane generated by $a,b \in \C^n,$ then $$ \mu_{_{U^n_1}}( \Pi)= (|a_1|^2 + |b_1|^2, \ldots , |a_n|^2 + |b_n|^2).$$
Then the image of the moment map $\mu_{_{U^n_1}} (Gr_{2,n})$ is the hypersimplex $\Xi$
$$\mu_{_{U^n_1}} (Gr_{2,n})= \Xi = \Big\{ (r_1, \ldots, r_n) \in \R^n | 0 \leq r_i \leq 1, \quad \sum_{i=1}^n r_i =2 \Big\} $$
and the set of critical values of $\mu_{_{U^n_1}}$ consists of those points $ (r_1, \ldots, r_n) \in \Xi$ satisfying one of the following conditions
\begin{itemize}
\item[a)] one of the $r_i$'s vanishes or is equal to 1;
\item[b)] there exists $\epsilon_i = \pm 1$ such that $\sum_{i=1}^n \epsilon_i r_i=0$ with at least two $\epsilon_i$'s for each sign.
\end{itemize}
Note that points satisfying a) constitute the boundary of $\Xi,$ while points satisfying condition b) are the inner walls of $\Xi.$ 

From the identification of the bending flows with the residual torus action coming from the Gel'fand--Cetlin system ( \cite{hk1} Theorem 5.2), Hausmann and Knutson prove that the action coordinates $\ell_1, \ldots, \ell_{n-3}$ satisfy the system 
\begin{equation}\label{GC}
\left\{\begin{array}{c}
r_{i+2} \leq \ell_i+ \ell_{i+1}\\
\ell_i \leq r_{i+2}+ \ell_{i+1}\\
\ell_{i+1} \leq r_{i+2} + \ell_i.\\
\end{array} \right.
\end{equation}

In the case $n=5 $ the choice of the two (proper) diagonals from the first vertex, i.e. $\mu_1= e_1+e_2$ and $\mu_2= e_1+e_2+e_3= -(e_4+e_5),$ allows us to define a toric bending action. The moment polytope $\mu_{T^2}(M_r)$ associated to this bending action is the intersection $\mu_{_{T^{2}}}(M_r)= I \cap \Upsilon$ where $I$ is the rectangle $$I = \Big[ |r_1-r_2|, r_1+r_2 \Big] \times \Big[ |r_4-r_5|, r_4+r_5 \Big]$$ and $\Upsilon$ is the region 
$$ \Upsilon = \{ (x,y ) \in \R^2 : y\geq -x+r_3; \, y \geq x-r_3; \, y \leq x+r_3 \}.$$

\begin{figure}[htbp]
\begin{center}
\psfrag{a}{\footnotesize{$y= | r_4-r_5|$}}
\psfrag{b}{\footnotesize{$y= r_4+r_5$}}
\psfrag{c}{\footnotesize{$x= | r_1-r_2|$}}
\psfrag{d}{\footnotesize{$x= r_1+r_2$}}
\psfrag{1}{\footnotesize{$y=x+r_3$}}
\psfrag{2}{\footnotesize{$y=-x+r_3$}}
\psfrag{3}{\footnotesize{$y=x-r_3$}}
\psfrag{x}{\footnotesize{$x$}}
\psfrag{y}{\footnotesize{$y$}}
\includegraphics[width=7cm]{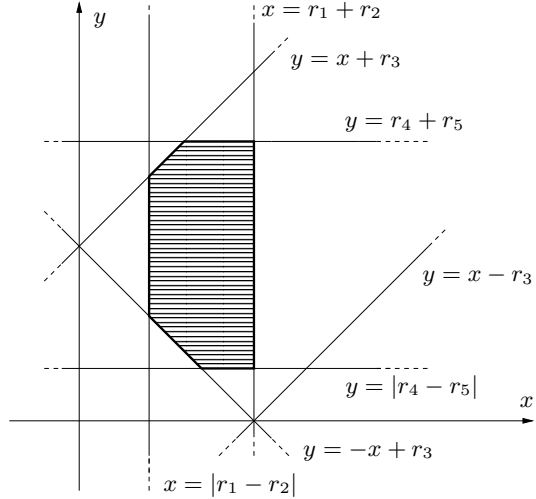}
\end{center}
\caption{$\mu_{{T^2}}(M_r)$}
\label{es0}
\end{figure}

For some examples we refer to section \ref{es_cob}.

\section{Cobordism of Polygon Spaces }
In Section \ref{cobordism} we state $S^1$-equivariant cobordism results due to Ginzburg, Guillemin and Karshon \cite{cobordism}. In Section \ref{dim} we apply these to the moduli space of polygons $M_r$ endowed with $S^1$-action of bending along a (chosen) diagonal.

\subsection{$S^1$-equivariant cobordism}\label{cobordism}
In this paper we investigate the $S^1$-cobordism class of the moduli space of polygons $M_r.$ Our proof will be based on Theorems \ref{theo.cob} and \ref{theo.cob2}) due to Ginzburg, Guillemin and Karshon \cite{cobordism} (see also \cite{ggk}). 
Martin \cite{Martin} also proved similar cobordism results.

\begin{theo} \label{theo.cob} \textnormal{(V.~Ginzburg, V.~Guillemin,
Y.~Karshon)}

 Let $M$ be an oriented $2d$-dimensional manifold
  on which the group $S^1$ acts. Suppose that this action is quasi-free
  and has finitely many fixed points. Then $M$ is cobordant a disjoint
  union of $N$ copies of $\pm \C\P^d,$ where $N$ is the number of fixed
  points.
\end{theo}
The proof (see \cite{cobordism}) shows that each isolated fixed point contributes to the cobordism class of $M$ with a copy of the projective space $\C \P^d.$ The orientation of this projective space comes from the infinitesimal generator of the bending action, thus might not agree with the standard one.

Both the assumptions on the action are extremely strong. If we do not ask the $S^1$ action to be quasi-free (but still to have finitely many fixed points) then it is still possible to prove a result on equivariant orbifold cobordism between $M$ and the disjoint union of twisted projective spaces (\cite{cobordism}, \cite{ggk}). On the other hand, if we assume the action to be quasi-free but we allow the fixed point set not to be finite, still it is possible to describe explicitly the equivariant cobordism class of $M.$

\begin{theo} \label{theo.cob2}  \textnormal{(V.~Ginzburg, V.~Guillemin,
Y.~Karshon)}

 Let $M$ be an oriented $2d$-dimensional manifold endowed with a quasi-free $S^1$ action. Let $X_k, k=1, \ldots, N,$ be the connected components of the fixed point set $M^{S^1}.$ Then 
$$ M \sim \coprod_{k=1}^N B_k, $$ where $B_k$ is a fibration over $X_k $ with fiber $\C\P^{m_k},$ and $m_k= codim_{\C} X_k.$
\end{theo}  

It is also possible to describe the equivariant orbifold cobordism class of $M$ when the $S^1$ action is not quasi-free and $M^{S^1}$ is not finite. In this more general case a result similar to Theorem \ref{theo.cob2}  holds, but the fibrations over the connected components of $M^{S^1}$ have now fibers which are twisted projective spaces.

\subsection{Proof of the Cobordism Theorem}\label{dim}
In light of the results presented in the previous sections we investigate the set of fixed points for a bending action.
Let  $\beta$ be the action of $S^1$ on $M_r$ by bending along the $(n-3)$-th diagonal $\mu_{(n-3)}= e_1+e_2+ \cdots + e_{n-2},$ i.e.
\begin{displaymath}
\begin{array}{crcl}
\!\! \beta:& \!S^1 \times M_r & \! \rightarrow & M_r\\
\!\! & \! (t, [(e_1, \ldots e_n)]) & \! \mapsto &\! \! \! [(\exp(t \mathrm{ad}_{\mu_{(n-3)}})e_1, \ldots, \exp(t \mathrm{ad}_{\mu_{(n-3)}})e_{n-2}, e_{n-1}, e_n)].
\end{array}
\end{displaymath}
The action $\beta$ is quasi-free, in fact the stabilizers of points are connected (they are $S^1$ for fixed points, $\{0\} $ otherwise). 

A point $P \in M_r$ is fixed by $\beta$ if it is of one of the following two types:
\begin{itemize}
\item[(I)] $[P]=[\vec{e}], e_1, \ldots, e_{n-2}$ are collinear as in Figure \ref{figura2}

\begin{figure}[htbp]
\begin{center}
\psfrag{1}{\footnotesize{$e_{1}$}}
\psfrag{2}{\footnotesize{$e_{2}$}}
\psfrag{k}{\footnotesize{$e_{n-2}$}}
\psfrag{k+1}{\footnotesize{$e_{n-1}$}}
\psfrag{++2}{\footnotesize{$e_{n}$}}
\includegraphics[width=7cm]{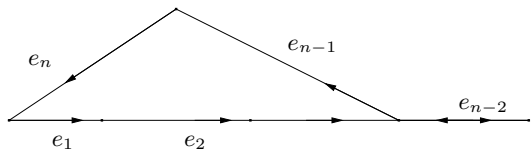}
\end{center}
\caption{Fixed point of type I}
\label{figura2}
\end{figure}

In this case the action $\beta$ fixes not just $[P]$ but also each representative.

\item[(II)] $[P]=[\vec{e}], e_{n-1}, e_{n}$ are collinear as in Figure \ref{figura3}.

\begin{figure}[htbp]
\begin{center}
\psfrag{1}{\footnotesize{$e_{1}$}}
\psfrag{2}{\footnotesize{$e_{2}$}}
\psfrag{n-1}{\footnotesize{$e_{n-1}$}}
\psfrag{n}{\footnotesize{$e_{n}$}}
\includegraphics[width=3cm]{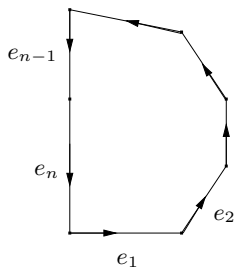}
\end{center}
\caption{Fixed point of type II}
\label{figura3}
\end{figure}

In this case the action $\beta$ changes the representative $\vec{e}$  but not the $SO(3)$ class.

\end{itemize} 
The fixed point set $ M_r^{S^1}$ is then the (disjoint) union of the sets $ (M_r^{S^1})_{isol}$ of fixed points of type I and $ (M_r^{S^1})_{subm}$ of fixed points of type II.

If $[P]$ is a fixed point of type II then $[P] \in X_k,$ where $X_k$ is a submanifold of fixed points. In particular $X_k$ is the space of polygons of $n-1$ sides $M_{\bar{r}},$ with $\bar{r}= (r_1, \ldots, r_{n-2}, \pm r_{n-1} \pm r_n) \in \R^{n-1}_+.$ (The signs $\pm$ are determined according to the orientation of the edges $e_{n-1}$ and $e_{n}$.) It follows that $codim_{\C} X_k=1,$ and so $X_k$ contributes to the cobordism of $M_r$ with the total space $B_k$ of a fibration on $X_k$ with fiber $\C\P^1.$ This implies that $B_k \sim 0$ because it is the boundary of the associated fibration $\tilde{B_k}$ on $X_k $ with fiber the disk $D$  $(\delta D = S^2 \sim \C\P^1).$

Fixed points of type I are instead isolated and so from Theorem \ref{theo.cob} contribute to the cobordism class of $M_r$ with a copy of $\C\P^{n-3}.$ The orientation of this projective space comes from the generator of the bending action and may not agree with the orientation that  $\C\P^{n-3}$ inherits from the symplectic structure of $M_r.$ In fact for each  $[P] \in (M_r^{S^1})_{isol}$ the symplectic form $\omega$ on $M_r$ defines a complex structure $J$ on $T_{[P]}M_r$ by $\omega_{[P]}(u,v)= g(u,Jv),$ where $g$ is a Riemannian metric on $\R^{3n}.$ The bending action defines too a complex structure on $T_{[P]}M_r:$ differentiating $\beta$ in $(\theta, [P])$ and valuating it at $1 \in \R \simeq Lie(S^1)$ we obtain an endomorphism of $T_{[p]}M_r$ and this defines also an $S^1$-action (the linear isotropy action)  on $T_{[p]}M_r:$
\begin{displaymath}
\begin{array}{crcl}
d_{[P]} \beta: & S^1 & \rightarrow & End(T_{[p]}M_r)\\
 & \theta & \mapsto & d_{(\theta, [P])} \beta (1)\\
\end{array}
\end{displaymath}
under which $T_{[p]}M_r$ decompose in the direct sum $$T_{[p]}M_r = \bigoplus_{w \in \Z}V_w$$ so that on each $V_w$ the $S^1$-action is ``multiplication by $e^{iw \theta}$''. The $w$'s are the isotropy weights and, because the action is semi-free (for $S^1$-actions quasi-free and semi free are equivalent), they are $0$ or $\pm1.$ The differential of $d_{[P]} \beta$
$$A= \frac{d}{d \theta} \big( d_{[P]} \beta  \big)_{|=0}(1) : T_{[p]}M_r \rightarrow T_{[p]}M_r$$ is the generator of the bending action (note that on each $V_w,$ $A$ is the multiplication by $ i w$).

To determine the cobordism class of $M_r$ we will calculate the orientation that $A$ induces on the projective spaces $\C\P^{n-3}.$  The proof will go as follows: first we will calculate $$\hat{A}= \frac{d}{d \theta} \big( d_{\vec{e}} \hat{\beta}  \big)_{|=0}(1) : T_{\vec{e}}\tilde{M}_r \rightarrow T_{\vec{e}} \tilde{M}_r$$ where $\hat{\beta} $ is the bending action on the level set $$ \tilde{M_r} = \{ \vec{e} \in \prod_{j=1}^{n} S^2 (r_j) / e_{1} + \ldots + e_{n} =0\},$$ i.e.
\begin{displaymath}
\begin{array}{crcl}
\hat{\beta}:& S^1 \times \tilde{M}_r & \rightarrow & \tilde{M}_r\\
 & (t, (e_1, \ldots e_n)) & \mapsto & (\exp(t\,\mathrm{ad}_{\mu_{(n-3)}})e_1, \ldots, \exp(t\,\mathrm{ad}_{\mu_{(n-3)}})e_{n-2}, e_{n-1}, e_n).
\end{array}
\end{displaymath}
Then identifying $T_{[P]} M_r$ with the orthogonal $T^{\perp}_{P} (SO(3)\cdot P)$ of tangent space to the $SO(3)$ orbit through $P$ in $\tilde{M_r}$ we will project $\hat{A}$ on $T_{[P]}M_r$ and write $A$ explicitly. Finally we will verify that $A$ is a complex structure and compare it with $J$ by checking when a $J$-positive basis of $T_P M_r $ is also $A$-positive.

\begin{rmk}
Observe that $\hat{A}$ is well defined because if $P=[\vec{e}]$ is a fixed point of type I then $\vec{e}$ is a fixed point for $\hat{\beta}$ (i.e. $\beta $ fixes each representative of the class, not just the class).
\end{rmk}
\subsubsection{The complex structure $A$}
{\bf Determining $ \hat{A} : T_P \tilde{M_r} \to T_P \tilde{M_r} $}

The action $\hat{\beta} $ described above still bends the first $(n-2)$ sides of a polygon along its $(n-3)$-diagonal.
An element of $T_P \tilde{M_r}$ is of the form $\frac{d}{d \epsilon} (P+ \epsilon Q)\vert_{\epsilon =0}$,
$P+ \epsilon Q = (e_1 +\epsilon v_1, \ldots, e_n+ \epsilon v_n).$ 
Let $\mu$ be the $(n-3)$-diagonal of the polygon $P$, i.e. $\mu = e_1+ \ldots + e_{n-2}$, and let $\nu$ be the $(n-3)$-diagonal of $P+ \epsilon Q$, i.e. $$\nu = \sum_{i=1}^{n-2} e_i + \epsilon \sum_{i=1}^{n-2} v_i := \mu + \epsilon \xi.$$ 
From now on, when $\vec{v}$ is understood, we will write $\xi$ for $\xi(\vec{v})= \sum_{i=1}^{n-2} v_i.$

Let $R_\epsilon $ be the rotation that takes $ \nu$ to the $x$-axis and let  $b_{\theta}$ be the rotation of angle $\theta$ around the 
$x$-axis. The bending action $\hat{\beta}$ can be described in terms of $R_{\eps}$ and $b_{\theta}$, precisely: $$\hat{\beta}(P+ \epsilon Q)=  (\ldots, R_{\epsilon}^{-1} b_{\theta} 
R_{\epsilon} (e_{j}+ \epsilon v_j), \ldots, e_{n-1} + \epsilon v_{n-1}, 
 e_{n} + \epsilon v_{n}). $$
So $$
\begin{array}{llll}
\hat{A} :& T_P \tilde{M_r} & \to & T_P \tilde{M_r} \\
 & v & \mapsto & \hat{A}(v)\\
\end{array} $$
with $$ \hat{A}(v) = \frac{d}{d \theta}_{\vert \theta =0}  
\frac{d}{d \epsilon}_{\vert \epsilon =0} (\ldots, R_{\epsilon}^{-1} b_{\theta} 
R_{\epsilon} (e_{j}+ \epsilon v_j), \ldots, e_{n-1} + \epsilon v_{n-1}, 
 e_{n} + \epsilon v_{n}). $$

\begin{rmk}\label{boh}
We will use the notation $\j \wedge \k $ for the matrix 
$\left(\begin{array}{ccc}
0&0&0\\
0&0&-1\\
0&1&0\\
\end{array}\right)$ of the rotation around the $x$-axis. In general, for $u_1, u_2 $ 
in $ \R^3, $ $u_1 \wedge u_2$ is the rotation which takes $u_1$ on $u_2,$ i.e.$$(u_1 \wedge u_2)(v) =   \langle u_1, v \rangle u_2 - \langle u_2, v\rangle u_1 \quad \forall v \in \R^3.$$
\end{rmk}

\begin{prop}
$$\frac{d}{d \theta}_{\vert \theta =0}  
\Big(\frac{d}{d \epsilon}_{\vert \epsilon =0} R_{\epsilon}^{-1} b_{\theta} R_{\epsilon} 
(e_{j}+ \epsilon v_j)\Big) = - \frac{ \langle \mu, e_j \rangle}{\| \mu \|^2} \underline{j} 
\wedge \underline{k} (\xi) + \underline{j} \wedge \underline{k} (v_j). $$
\end{prop}

\begin{proof}
Using the same notation as in Remark \ref{boh}, the rotation  $R_{\epsilon}$ is $\exp (- \Theta \frac{\mu \wedge \epsilon \xi}{\lVert \mu \rVert \lVert \epsilon \xi \rVert}),$ 
where the angle of  rotation is 
$\Theta = \frac{\lVert \epsilon \xi \rVert}{\lVert \mu \rVert},$ and $b_{\theta}$ is $ \exp (\theta  \underline{j} \wedge \underline{k}).$ The first order Taylor expansions of $R_{\epsilon}^{-1}$ and $ R_{\epsilon}$ are $$R_{\epsilon}^{-1} 
= id + \epsilon \frac{\mu \wedge \xi}{\lVert \mu \rVert^2} + o (\epsilon), \quad \quad R_{\epsilon}
= id - \epsilon \frac{\mu \wedge \xi}{\lVert \mu \rVert^2} + o (\epsilon)$$


So $$  \frac{d}{d \epsilon}_{\vert \epsilon =0} R_{\epsilon}^{-1} b_{\theta} 
R_{\epsilon} (e_{j}+ \epsilon v_j)= \Big[\frac{\mu \wedge \xi}{\| \mu \|^2} b_{\theta}(e_j+ \epsilon v_j)-  b_{\theta}\frac{\mu \wedge \xi}{\| \mu \|^2} (e_j+ \epsilon v_j) +b_{\theta} v_j \Big]_{|\epsilon=0}=$$
$$  \frac{\mu \wedge \xi}{\| \mu \|^2} b_{\theta} e_j - b_{\theta} \frac{\mu \wedge \xi}{\| \mu \|^2} e_j + b_{\theta} v_j.$$
Similarly observe that the first order Taylor expansion of $b_{\theta}$ is
$$ b_{\theta} = id + \theta \underline{j} \wedge \underline{k} + o(\theta),$$ and so
$$\frac{d}{d \theta}_{\vert \theta =0} \frac{d}{d \epsilon}_{\vert \epsilon =0}  R_{\epsilon}^{-1} b_{\theta} R_{\epsilon} (e_{j}+ \epsilon v_j)=\frac{d}{d \theta}_{\vert \theta =0}\Big( \frac{\mu \wedge \xi}{\| \mu \|^2} b_{\theta} e_j - b_{\theta} \frac{\mu \wedge \xi}{\| \mu \|^2} e_j + b_{\theta} v_j\Big)=$$

$$  \frac{\mu \wedge \xi}{\| \mu \|^2}  \underline{j} \wedge \underline{k}( e_j)
- \underline{j} \wedge \underline{k} \frac{\mu \wedge \xi}{\| \mu \|^2} ( e_j) +\underline{j} \wedge \underline{k}( v_j)=$$

$$  \frac{\mu \wedge \xi}{\| \mu \|^2}( \underbrace{ \langle \underline{j}, e_j \rangle}_{=0}  \underline{k} -  \underbrace{ \langle \underline{k}, e_j\rangle}_{=0}  \underline{j})-
\frac{\underline{j} \wedge \underline{k}}{\| \mu \|^2}( \langle \mu, e_j \rangle \xi - 
\underbrace{ \langle \xi, e_j\rangle}_{=0} \mu) + \j \wedge \k (v_j)= $$

$$- \frac{ \langle \mu, e_j\rangle}{\| \mu \|^2} \j \wedge \k (\xi) + \j \wedge \k (v_j). $$
\end{proof}

Hence the map $\hat{A}$ is given by
$$
\begin{array}{llll}
\hat{A} :& T_P \tilde{M_r} & \to & T_P \tilde{M_r} \\
 & v & \mapsto & (\hat{A}_1(v), \ldots, \hat{A}_k (v), 0,0)= \hat{A}(v),\\
\end{array} $$
where \begin{equation}\label{Acap} \hat{A}_j (v)= - \frac{ \langle \mu, e_j \rangle}{\| \mu \|^2} \underline{j} \wedge \underline{k} (\xi) + \underline{j} \wedge \underline{k} (v_j).\end{equation}

{\bf Passage to the quotient  $M_r = \tilde{M_r} \Big/ SO(3)$}\\
Under the $SO(3)$-action the tangent space in $P$ at $\tilde{M}_r$ decomposes in the direct sum of the tangent space at the $SO(3)$ orbit trough $P$ and its orthogonal: 
$$T_{ P}\tilde{M}_r= T_P (SO(3) \cdot P) \oplus T_P^{\perp} (SO(3) \cdot P).  $$\label{identif}
Identifying  $ T_P^{\perp} (SO(3) \cdot P)$ with $ T_{[ P]}M_r$ we calculate $A$ by  projecting $\hat{A}$ on $T_P^{\perp} (SO(3) \cdot P),$ i.e., if $\delta^1, \delta^2, \delta^3$ is an orthogonal basis of $T_p (SO(3) \cdot P),$

\begin{equation} \label{proiez} A(v) = \hat{A}(v) - \frac{\langle \hat{A}(v), \delta^1 \rangle}{\| \delta^1 \|^2 } \delta^1 -  \frac{\langle \hat{A}(v), \delta^2 \rangle}{\| \delta^2 \|^2 } \delta^2 -
 \frac{\langle \hat{A}(v), \delta^3 \rangle}{\| \delta^3 \|^2 } \delta^3.\end{equation}

The generators of the $SO(3)$-action are the rotations around the axes.
So $\hat{\delta}^1 =(e_1 \wedge \i, \ldots, e_n \wedge \i),
\hat{\delta}^2 =(e_1 \wedge \j, \ldots, e_n \wedge \j),
\hat{\delta}^3 =(e_1 \wedge \k, \ldots, e_n \wedge \k)$
define a basis of $T_p (SO(3) \cdot P).$ This basis in general is not orthogonal with respect to the metric associated to the symplectic structure and we will orthonormalize it using the Gram-Schmidt formula. So, in order to write explicitly the basis  $\hat{\delta}^1, \hat{\delta}^2 $ and $\hat{\delta}^3$ of the $SO(3)$-orbit trough $P$ in $\tilde{M_r}$ let us fix a representative $\vec{e}$ in $[P].$

Because $P$ is planar it is not restrictive to assume that it lies in the plane $(x,y).$ Moreover, let us assume that the coordinate axis $x$ is oriented as the $(n-3)$-th diagonal $\mu_{(n-3)} := \mu,$ then the triangle in Figure \ref{figura2} has side lengths $r_n, r_{n-1}, $ and $\sum \epsilon_i r_i, $ where $\epsilon_1=1 $ if $e_i= \frac{r_i}{\| \mu \|} \mu$ and $\epsilon_i = - 1$ otherwise. This gives a geometric interpretation of the notion of $r$-admissibility for an index set $I$ introduced in Definition \eqref{df}. In fact $I$ counts the number of ``forward tracks'', or, more formally, if $\ell= |I|,$ then $$\ell= \sharp \{ e_j / e_j \cdot \mu > 0 \} $$
and the inequalities in system \eqref{tr} are just the ``triangle inequalities'' for the triangle of edge lengths  $r_n, r_{n-1}, $ and $\sum \epsilon_i r_i $. So $I$ is $r$-admissible if and only if such a triangle (as in Figure \ref{figura2}) closes.
The assumptions done so far are not restrictive.
Let us also assume that the first $\ell$ edges are oriented as the $x$-axis, i.e.
\begin{equation}\label{ipotesi1} e_i = (r_i, 0,0),\quad \forall i=1, \ldots, \ell,\end{equation}
and that the following $(n-2- \ell)$ edges are conversely oriented, i.e. 
\begin{equation}\label{ipotesi2} e_i =( -r_i, 0,0) \quad \forall i= \ell+1, \ldots, n-2.\end{equation}
This assumption is instead restrictive, we are in fact choosing the polygon $P$ corresponding to the index set $I=\{1, \ldots, \ell \}.$ This assumption is useful in order to keep the notation more compact. In Remark \ref{l} we will say some more words about what happens if we consider another class.

\begin{figure}[htbp]
\begin{center}
\psfrag{1}{\footnotesize{$e_{1}$}}
\psfrag{2}{\footnotesize{$e_{2}$}}
\psfrag{l}{\footnotesize{$e_{\ell}$}}
\psfrag{l+1}{\footnotesize{$e_{\ell+1}$}}
\psfrag{k}{\footnotesize{$e_{n-2}$}}
\psfrag{k+1}{\footnotesize{$e_{n-1}$}}
\psfrag{++2}{\footnotesize{$e_{n}$}}
\psfrag{a}{\footnotesize{$\alpha$}}
\psfrag{b}{\footnotesize{$\theta$}}
\includegraphics[width=7cm]{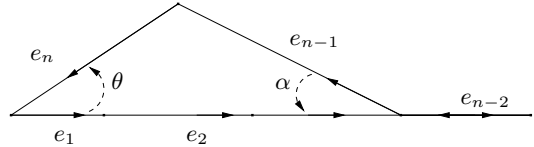}
\end{center}
\caption{ Model for $[P]$ fixed point of type I}
\label{figura1}
\end{figure}

Under these assumptions the polygon $P$ is as in Figure \ref{figura1} the last two edges $e_n$ and $e_{n-1}$ are
$$e_{n}= (- r_{n} \cos \theta,- r_{n} \sin \theta,0  ), \quad 
e_{n-1}= (- r_{n-1} \cos \alpha, r_{n-1} \sin \alpha,0  ).$$
We can express $\cos \alpha$ and $\sin \alpha$ as functions of $\theta, r_{n-1}$ and $r_{n}$ as follows:
$$\sin \alpha = \frac{r_{n}}{r_{n-1}} \sin \theta,\quad \cos \alpha = \frac{\| \mu \| - r_{n} \cos \theta}{r_{n-1} }.$$
With these assumptions the vectors  $\hat{\delta}^1,$  $\hat{\delta}^2$ and  $\hat{\delta}^3$ are 
\begin{center} 
$\hat{\delta}^1 = (0, \ldots, 0 , - r_{n-1} \sin \alpha \, \k,  r_{n} \sin \theta \, \k );$\end{center}
\begin{center}$ \hat{\delta}^2 = (r_1 \, \k, \ldots, r_{\ell} \, \k, - r_{\ell+1} \, \k , \ldots, - r_{n-2} \, \k , - r_{n-1} \cos \alpha \, \k,  -r_{n} \cos \theta \, \k );$\end{center}
$\hat{\delta}^3= (-r_1 \, \k, \ldots, -r_{\ell} \, \k,  r_{\ell+1} \, \k , \ldots,  r_{n-2} \, \k , 
r_{n-1} \sin \alpha \, \i + r_{n-1} \cos \alpha \, \j,$
\begin{flushright} $ -r_{n} \sin \theta \, \i + r_{n} \cos \theta \, \j).$ \end{flushright}

Applying Gram-Schmidt we build an orthogonal basis $\{ \delta^1, \delta^2, \delta^3 \}$ from the basis $\{ \hat{\delta}^1, \hat{\delta}^2, \hat{\delta}^3\},$ i.e. $\delta^1 := \hat{\delta^1},$ 
$\delta^2 := \hat{\delta^2}- \frac{\langle \hat{\delta}^2,\delta^1\rangle}{\langle \delta^1,\delta^1 \rangle }\delta^1$ and 
$\delta^3 := \hat{\delta^3}- \frac{\langle \hat{\delta}^3,\delta^1 \rangle }{\langle \hat{\delta^1},\hat{\delta^1} \rangle } \hat{\delta^1} - \frac{\langle \hat{\delta}^3,\delta^2 \rangle }{\langle \delta^2,\delta^2 \rangle } \delta^2.$

Recall that the scalar product on $T_P \tilde{M}_r$ is $\langle u,v \rangle = \sum_{i=1}^{n} \frac{1}{r_i} \langle u_i, v_i \rangle_S$ where $\langle \cdot, \cdot \rangle_S $ is the standard scalar product in $\R^3.$ So 
$$ \langle \hat{\delta^2},\hat{\delta^1} \rangle = \frac{r_{n}}{r_{n-1}} \sin^2 \theta (r_{n-1} +r_{n} ) \quad \textnormal{and} \quad \langle \hat{\delta^1},\hat{\delta^1}\rangle =  \frac{r_{n}}{r_{n-1}} \sin^2 \theta (r_{n-1} +r_{n} ).$$
Moreover  $\langle \hat{\delta}^3,\delta^1\rangle =0 $ and $\langle \hat{\delta}^3,\delta^2 \rangle =0.$

 To summarize, an orthogonal basis of $ T_p (SO(3) \cdot P)$ is given by:
$$\delta^1 = (0, \ldots, 0 , - r_{n} \sin \theta \k,  r_{n} \sin \theta \k ),$$
$$ \delta^2 = (r_1 \k, \ldots, r_{\ell} \k, - r_{\ell+1} \k , \ldots, - r_{n-2} \k , - \frac{r_{n-1} \| \mu \|}{r_{n-1} + r_{n}} \k,  - \frac{r_{n} \| \mu \|}{r_{n-1} + r_{n}} \k ),$$
$ \delta^3 = (-r_1 \, \k, \ldots, -r_{\ell} \, \k,  r_{\ell+1} \, \k , \ldots,  r_{n-2} \, \k ,r_{n} \sin \theta \, \i +( \| \mu \| - r_{n} \cos \theta) \, \j,$
\begin{flushright}$  -r_{n} \sin \theta \, \i + r_{n} \cos \theta \, \j).$ \end{flushright}

{\bf Computing $A(v)$}\\

Recall that $ A(v) = \hat{A}(v) - \frac{ \langle \hat{A}(v), \delta^1\rangle}{\| \delta^1 \|^2 } \delta^1 -  \frac{ \langle \hat{A}(v), \delta^2\rangle }{\| \delta^2 \|^2 } \delta^2 -
 \frac{\langle \hat{A}(v), \delta^3 \rangle }{\| \delta^3 \|^2 } \delta^3, $
where from \eqref{Acap} the $j$-th component of $\hat{A}(v)$ is $$\hat{A}(v)_j=  \Big( - \frac{\langle \mu, e_j \rangle }{\| \mu \|^2} \langle \underline{j}, \xi \rangle +  \langle \underline{j}, v_j \rangle \Big) \underline{k} +\Big(  \frac{\langle \mu, e_j \rangle }{\| \mu \|^2} \langle \underline{k}, \xi \rangle -  \langle \underline{k}, v_j \rangle \Big) \underline{j}$$ for $j=1, \ldots, n-2,$
and $\hat{A}(v)_j=0$ if $j=n-1,n.$

Let $\epsilon_j$ denote the direction of $e_j,$ i.e.
\begin{displaymath}
\epsilon_j= \left\{ \begin{array}{ll} 
1,& j=1, \ldots, \ell\\
-1& j= \ell+1, \ldots,n-2.\\
\end{array} \right.
\end{displaymath} 
It is straightforward to verify that $\langle \hat{A}(v), \delta^1\rangle = 0,$
$$\langle \hat{A}(v), \delta^2\rangle = \sum_{j=1}^{n-2} \epsilon_j \Big( - \frac{\langle \mu, e_j \rangle }{\| \mu \|^2} \langle \j,\xi \rangle  + \langle \j,v_j \rangle \Big) $$
and
$$\langle \hat{A}(v), \delta^3 \rangle = \sum_{j=1}^{n-2} (- \epsilon_j) \Big( - \frac{\langle \mu, e_j \rangle }{\| \mu \|^2} \langle \k,\xi \rangle + \langle \k,v_j \rangle  \Big).$$

Hence, for each $v \in T_{[P]} M_r$ the components of $A(v)$ are:

$$A(v)_j =  \Big( - \frac{r_j}{\| \mu \|} \langle \j,\xi \rangle + \langle \j,v_j\rangle -  \frac{\langle \hat{A}(v), \delta^2 \rangle}{\| \delta^2 \|^2 } r_j\Big) \k + \Big( \frac{r_j}{\| \mu \| } \langle \k, \xi \rangle - \langle \k, v_j \rangle + \frac{\langle \hat{A}(v), \delta^3 \rangle}{\| \delta^3 \|^2 } r_j \Big) \j,$$ for all $j=1, \ldots, \ell ;$
$$A(v)_j= \Big(  \frac{r_j}{\| \mu \|} \langle \j,\xi \rangle + \langle \j,v_j \rangle +  \frac{\langle \hat{A}(v), \delta^2 \rangle}{\| \delta^2 \|^2 } r_j \Big) \k + \Big(  - \frac{r_j}{\| \mu \| } \langle \k, \xi \rangle - \langle \k, v_j \rangle - \frac{\langle \hat{A}(v), \delta^3 \rangle}{\| \delta^3 \|^2 } r_j \Big)\j,$$
for all $j=\ell+1, \ldots, n-2; $
\begin{equation}\label{lis1}\!\!A(v)_{_{n-1}} \!\!= \!  \frac{\langle \hat{A}(v), \delta^2 \rangle}{\| \delta^2 \|^2 } \frac{r_{n-1} \| \mu \| }{r_{n-1} + r_{n} } \k -  \frac{\langle \hat{A}(v), \delta^3 \rangle}{\| \delta^3 \|^2}
(r_{n} \sin \theta \i + (\| \mu \| - r_{n} \cos \theta))\j ;\end{equation}

\begin{equation}\label{lis2}A(v)_{_{n}}=  \frac{\langle \hat{A}(v), \delta^2\rangle}{\| \delta^2 \|^2 } \frac{r_{n} \| \mu \| }{r_{n-1} + r_{n} } \k -  \frac{\langle \hat{A}(v), \delta^3 \rangle}{\| \delta^3 \|^2} 
(-r_{n} \sin \theta \i +  r_{n} \cos \theta)\j .\end{equation}

\subsubsection{Comparing the complex structures $A$ and $J$}

{\bf Determining a basis for $T_{[P]} M_r$}\\
Using the identification $ T_{[P]} M_r \simeq T_P^{\perp} (SO(3) \cdot P), $ Kapovich and Millson (\cite{km}) write the equations of $T_{[P]} M_r$ as a subspace of $\R^{3n}.$ Precisely, $v \in T_{[P]} M_r$ if and only if:
\begin{itemize}
\item[i)] $\sum_{i=1}^{n} v_i = 0,$
\item[ii)] $ e_i \cdot v_i =0 \quad \forall i=1, \ldots , n,$
\item[iii)] $\sum_{i=1}^{n} \frac{1}{r_i} (e_i \wedge v_i)= 0.$
\end{itemize}

The vectors 
\begin{displaymath}
\begin{array}{cl}
u_i = (0, \ldots, 0,\underbrace{ \underline{j}}_{i},\underbrace{- \underline{j}}_{i+1}, 0, \ldots, 0),&\quad i=1, \ldots, \ell-1,\\
\hat{u}_i = (0, \ldots, 0,\underbrace{ \underline{j}}_{i},\underbrace{- \underline{j}}_{i+1}, 0, \ldots, 0), &\quad i=\ell+1, \ldots,n-3, \\
v_i = (0, \ldots, 0,\underbrace{ \underline{k}}_{i},\underbrace{- \underline{k}}_{i+1}, 0, \ldots, 0), &\quad i=1, \ldots, \ell-1,\\
\hat{v}_i = (0, \ldots, 0,\underbrace{ \underline{k}}_{i},\underbrace{- \underline{k}}_{i+1}, 0, \ldots, 0), &\quad i=\ell, \ldots,n-3 \\
\end{array}
\end{displaymath}
verify the conditions $ i),$ $ ii),$ $ iii),$ so they are in $ T_{[P]}M_r,$ and are linearly independent.

\begin{rmk}
Note that a vector of the form $ (0, \ldots, 0,\underbrace{ \underline{j}}_{\ell},\underbrace{- \underline{j}}_{\ell+1}, 0, \ldots, 0)$ would not satisfy condition iii).
\end{rmk}

When $\ell=n-2$ we have $2(n-3)$such vectors and they are linearly independent,  forming a basis of  $T_{[P]} M_r.$\\
If instead $\ell\ne n-2 $ then we have $2 (n-4)$ vectors and it is necessary to complete them to a basis. To do this we look for a vector of the form 
$$ w = ( \lambda \k, \ldots, \lambda \k, \gamma \k, \ldots, \gamma \k, \lambda_{n-1} \k, \lambda_{n} \k),$$
with $\lambda, \gamma,  \lambda_{n-1},  \lambda_{n} \in \R,$ and we impose that $w$ satisfies conditions i), ii) and  iii).
Condition $iii)$ is straightforward verified by $w.$
Condition $i)$ holds if and only if 
\begin{equation} \label{cond1}
\ell \la + (n-\ell-2) \gamma+ \la_{n-1}+ \la_{n} =0.
\end{equation}

Denoting by $w_i$ the $i$-th component of $w$
$$\frac{e_i}{r_i} \wedge w_i = (1,0,0) \wedge (0,0,\la) = - \la \j \quad \quad \forall i=1, \ldots, \ell,$$ 
$$\frac{e_i}{r_i} \wedge w_i = (-1,0,0) \wedge (0,0,\gamma) = \gamma \j \quad \quad \forall i=\ell+1, \ldots, n-2,$$

$$\frac{1}{r_{n-1}}e_{n-1} \wedge w_{n-1} =
 \Big( - \frac{\| \mu \| - r_{n} \cos \theta}{r_{n-1}},\frac{r_{n}}{r_{n-1}} \sin \theta,0 \Big) \wedge (0,0, \la_{n-1})$$
 $$ = \Big(\la_{n-1} \frac{r_{n}}{r_{n-1}} \sin \theta , \la_{n-1}  \frac{\| \mu \| - r_{n} \cos \theta}{r_{n-1}},0 \Big),$$
$$\frac{1}{r_{n}} e_{n} \wedge w_{n} =( -  \cos \theta, -  \sin \theta, 0 ) \wedge (0,0, \la_{n})=
 (- \la_{n}  \sin \theta , \la_{n} \cos \theta,0).$$
Consequently we obtain that condition iii) holds if and only if 
\begin{equation} \label{cond2}
 - \ell \lambda + (n-\ell-2) \gamma +\la_{n-1} \frac{\| \mu \| - r_{n} \cos \theta}{r_{n-1}} + \la_{n} \cos \theta = 0 
\end{equation}
and
\begin{equation} \label{cond3} \la_{n-1} \frac{r_{n}}{r_{n-1}} sen \theta - \la_{n} sin \theta =0. \end{equation}

So $w$ is determined by the system of equations (\ref{cond1}), (\ref{cond2}), (\ref{cond3}).
A solution of this system is 
$$  \la= - \frac{1}{2\ell} (\| \mu \| - r_{n-1}- r_{n}), \quad \gamma = \frac{1}{2(n- \ell-2)} (\| \mu \| + r_{n-1}+ r_{n})$$ $$ \la_{n-1}= - r_{n-1}, \quad  \la_{n} = - r_{n}.$$

From now on let us fix these values for $\la, \gamma, \la_{n-1}, \la_{n}.$
The vector $  w  $ is linearly independent with the vectors $u_i, \hat{u}_i, v_i, \hat{v}_i.$ $J$ is the complex structure associated to the symplectic form, so $-J(w)$ is linearly independent with  $u_i, \hat{u}_i, v_i, \hat{v}_i, w$ and complete to a basis of $T_{[P]}M_r.$
Recalling that $J(w) = (\frac{e_1}{r_1} \wedge w_1, \ldots,\frac{e_{n}}{r_{n}} \wedge w_{n} )$ we get
$$ -J(w) = (\lambda \j, \dots, \lambda \j, - \gamma \j, - \gamma \j, r_{n} \sin \theta \i + (\| \mu \| - r_{n} \cos \theta) \j , -r_{n} \sin \theta \i + 
 r_{n} \cos \theta \j).$$
So $\mathcal{B}_1 = \{ u_1, v_1, \ldots,u_{\ell-1}, v_{\ell-1}, \hat{u}_{\ell+1},- \hat{v}_{\ell+1}, \ldots,  \hat{u}_{n-3}, - \hat{v}_{n-3}, J(w) , w\}$ is a basis of  $T_{[P]}M_r$ and it is positive, i.e. this is the standard convention. In fact, 
$$J (u_i) = ( \ldots, \frac{r_i}{r_i} \i \wedge \j, \frac{r_i}{r_i} \i \wedge (- \j),0, \ldots, 0) = (0, \ldots, 0, \k, -\k, 0, \ldots, 0)= v_i,$$
$$J (\hat{u}_i) = ( \ldots,- \frac{r_i}{r_i} \i \wedge \j,- \frac{r_i}{r_i} \i \wedge (- \j),0, \ldots, 0) = (0, \ldots, 0,- \k, \k, 0, \ldots, 0)= - \hat{v}_i,$$
and  $J (v_i) = -u_i,$  $J (- \hat{v}_i) = - \hat{u}_i$ and  $J (-J(w)) =w.$

{\bf $A $ is a complex structure}\\
In this section we will verify that $A$ is a complex structure. To check that $A^2 = - Id$ we write the matrix of $A$ with respect to the basis $\mathcal{B}_1$ ( with a little abuse of notation, we will call this matrix $A$).

First of all we can note that $\xi(u_i)= \xi(\hat{u}_i)= \xi(v_i)= \xi(\hat{v}_i)=0$ (remember that $\xi(v)= \sum_{i=1}^{n-2} v_i$ for all $v \in \R^{3n}$). So $$\langle \hat{A}(u_i), \delta^2 \rangle = \langle \j, \j \rangle + \langle \j , -\j \rangle =0$$ and similarly $\langle \hat{A}(\hat{u}_i), \delta^2 \rangle = \langle \hat{A}(v_i), \delta^3 \rangle = \langle \hat{A}(\hat{v}_i), \delta^3 \rangle = 0.$ Moreover it is trivial to see that $\langle \hat{A}(u_i), \delta^3 \rangle= \langle \hat{A}(\hat{u}_i), \delta^3 \rangle = \langle \hat{A}(v_i), \delta^2 \rangle= \langle \hat{A}(\hat{v}_i), \delta^2 \rangle = 0.$
Now it is easy to verify that 
$$A (u_i) =  (0, \ldots, 0,  \underline{k}, -\underline{k},0, \ldots, 0)= v_i, \quad \forall i=1, \ldots \ell-1,$$
$$A (v_i) =  (0, \ldots, 0,  -\underline{j}, \underline{j},0, \ldots, 0)= - u_i, \quad \forall i=1, \ldots \ell-1,$$
 $$A (\hat{u}_i) =(0, \ldots, 0,  \underline{k}, -\underline{k},0, \ldots, 0)=  \hat{v}_i, \quad \forall i=\ell, \ldots n-3,$$
 $$A (\hat{v}_i) =(0, \ldots, 0, - \underline{j}, \underline{j},0, \ldots, 0)=-  \hat{u}_i, \quad \forall i=\ell, \ldots n-3,$$
Also  $$ A ( -J(w)) = b_1 v_1 + \ldots + b_{k-2} \hat{v}_{k-2} + b w$$ and $$A(w) = a_1 u_1 + \ldots + a_{n-3} \hat{u}_{n-3} +a (-J(w))$$  $a_i, b_i,a,b \in \R,$ and so the matrix $A$ is:
{\small
\begin{displaymath}
A= \left( \begin{array}{ccccc|ccccc|cc}
0&1& & & & & & & & & 0& a_1\\
-1 & 0& & & & & & & & & b_1&0\\
 & & \ddots & & & & & & & & \vdots & \vdots\\
 & & & 0 & 1& & & & & & 0 & a_{\ell-1}\\
 & & & -1 &0 & & & & & & b_{\ell-1} & 0\\
\hline
& &  & & & 0 & -1 &  & & & 0 & a_{\ell}\\
 & & & & & 1 &0 & & & & b_{\ell}& 0\\
& & & & & & & \ddots& & &\vdots & \vdots \\
 & & & & & & & & 0 & -1& 0 & a_{n-3}\\
& & & & & & & & 1& 0 & b_{n-3} & 0\\
\hline
 & & & & & & & & & & 0& a\\
& & &  & && &  & &  & b & 0\\
\end{array} \right).
\end{displaymath}}  

Hence $A^2= - Id \iff ab =-1.$

{\bf Determine $a$ and $b$.}
First of all we can notice that the last two components of $A(-J(w))$ and $A(w)$ are enough to determine $a$ and $b$ because the vectors $u_i, \hat{u}_i, v_i, \hat{v}_i$ have no influence on the final components.\\
Observing that $\langle \hat{A}(-J(w)), \delta^3 \rangle = 0$ (because $-J(w)$ has no nonzero components along $\k$), it follows from \eqref{lis1} and \eqref{lis2} that:
$$ A(-J(w))= \Big( \ldots, \frac{\langle \hat{A} (-J(w)), \delta^2 \rangle}{\| \delta^2 \| ^2} \frac{r_{n-1} \| \mu \| }{r_{n-1} + r_{n}} \k,  \frac{\langle \hat{A} (-J(w)), \delta^2 \rangle}{\| \delta^2 \| ^2} \frac{r_{n} \| \mu \| }{r_{n-1} + r_{n}} \k \Big).$$
Now, recalling that $w = ( \lambda \k, \ldots, \lambda \k, \gamma \k, \ldots, \gamma \k, -r_{n-1} \k, -r_{n} \k)$ we get 
\begin{equation}\label{b}
b= -  \frac{\langle \hat{A} (-Jw), \delta^2 \rangle}{\| \delta^2 \| ^2} \frac{ \| \mu \| }{r_{n-1} + r_{n}}.
\end{equation} 

Similarly it is possible to observe that $\langle \hat{A} (w), \delta^2\rangle =0 ,$ thus

$$ A(w) = \Big( \ldots , - \frac{\langle \hat{A}(w), \delta^3 \rangle}{\| \delta ^3 \|^2 } \big( r_{n} \sin \theta \i + ( \| \mu \| - r_{n} \cos \theta) \j \big),$$ \vspace{-2mm}
$$ \quad \quad \quad \quad \quad \quad - \frac{\langle \hat{A}(w), \delta^3\rangle}{\| \delta ^3 \|^2 } \big( - r_{n} \sin \theta \i +  r_{n} \cos \theta \j \big) \Big).$$
Comparing $A(w)$ with the last two components of $-Jw$ we get:
\begin{equation}\label{a}
a= -  \frac{\langle \hat{A} (w), \delta^3\rangle}{\| \delta^3 \| ^2}.
\end{equation}

Remember that $\xi(-J(w))= \sum_{i=1}^{n-2} (-Jw)_i = \ell \lambda - (n-\ell-2) \gamma = - \| \mu \| \j.$
Then $$ \langle \hat{A} (-J(w)), \delta^2 \rangle = \sum_{j=1}^{\ell} \Big( \frac{r_j}{ \| \mu \|} \| \mu \| + \la \Big) - \sum_{j= \ell+1}^{n-2} \Big(- \frac{r_j}{ \| \mu \|} \| \mu \| - \gamma \Big)=$$
$$ \sum_{j=1}^{n-2} r_j + \ell \la + (n-\ell-2) \gamma =  \sum_{j=1}^{n-2} r_j +r_{n-1} + r_{n}= 2.$$

$$ \| \delta^2 \|^2 =  \sum_{j=1}^{n-2} r_j + \frac{\| \mu \|^2 (r_{n-1} + r_{n})}{(r_{n-1}+ r_{n})^2} = \frac{(r_{n-1}+ r_{n}) \sum_{j=1}^{n-2} r_j + \| \mu \|^2}{r_{n-1}+ r_{n}}.$$

So $$b = - \frac{2 \| \mu \|}{(r_{n-1}+ r_{n}) \sum_{j=1}^{n-2} r_j + \| \mu \|^2}.$$

Similarly, $ \xi (w) = \ell \lambda + (n-\ell -2) \gamma = - \frac{1}{2} (\| \mu \| - r_{n-1} - r_{n})+ \frac{1}{2} (\| \mu \| + r_{n-1} + r_{n})= r_{n-1} + r_{n}.$
$$ \langle \hat{A}(w), \delta^3 \rangle = \sum_{j=1}^{\ell} \Big( - \frac{r_j}{ \| \mu \|} (r_{n-1}+ r_{n}) + \la \Big) + \sum_{j=\ell+1}^{n-2} \Big( - \frac{r_j}{ \| \mu \|} (r_{n-1}+ r_{n}) - \gamma \Big) = $$
$$ - \frac{r_{n-1} + r_{n}}{\| \mu \|}  \sum_{j=1}^{n-2} r_j + \ell \la - (n-\ell-2) \gamma = - \frac{(r_{n-1} + r_{n})  \sum_{j=1}^{n-2} r_j + \| \mu \|^2}{\| \mu \|}.$$

$\| \delta^3 \|^2 = \sum_{j=1}^{n-2} r_j +r_{n-1} + r_{n}=2.  $\\
So 

\begin{equation}\label{a1}
a= \frac{(r_{n-1} + r_{n})  \sum_{j=1}^{n-2} r_j + \| \mu \|^2}{2 \| \mu \|}.
\end{equation} 

It is now straightforward to verify that $ab =-1,$ and so $A^2= - Id.$

\subsubsection{Conclusions}
$\mathcal{B}_2 = \{ \ldots, 
u_i, Au_i, \ldots, \hat{u_i}, A \hat{u}_i,\ldots, -Aw, w \}$ is a $A$-positive basis of $T_{[P]}M_r.$
Then $\mathcal{B}_2 $ is also $J$-positive if and only if the determinant of the matrix of the change of base $M_{\mathcal{B}_2\mathcal{B}_1}=M$ is positive. In this case the orientation induced by $ A$ is positive (or concord with the one induced by $J$).

Let $-Aw = \alpha_1 u_1 + \ldots + \alpha_{n} \hat{u}_{n} + \alpha (- Jw).$ From the description of $A$ given in the previous section the coordinate change matrix  is 
{\small
\begin{displaymath}
M = \left(\begin{array}{ccccc|ccccc|cc}

1 & & & & & & & & & & \alpha_1 & 0 \\
 & 1& & & & & & & & & 0 & 0 \\
 & & 1 & & & & & & & & \alpha_2 & 0 \\
 & & & \ddots & & & & & & & \vdots & \vdots\\
 & & & & 1& & & & & & 0& 0\\
\hline
 & & & & & 1& & & & & \alpha_{\ell} & 0\\
 & & & & & & -1& & & & 0 &0\\
 & & & & & & & \ddots & & & \vdots & \vdots\\
 & & & & & & & & 1&  & \alpha_{n-4} & 0\\
 & & & & & & & & &  -1 & 0 &0 \\
\hline
 & & & & & & & & & & \alpha & 0\\
 & & & & & & & & & & 0 & 1\\ 

\end{array} \right)
\end{displaymath}}

So $det M = (-1)^{n-3-\ell} \alpha.$

Now, since $\alpha = -a,$  it follows from \eqref{a1} that $$\alpha= - \frac{(r_{n-1} + r_{n})  \sum_{j=1}^{n-2} r_j + \| \mu \|^2}{2 \| \mu \|} < 0. $$
So  sgn$(\det(M)) = (-1)^{n-\ell}$ and $[P]$ contributes to the cobordism class of $M_r$ with $ (-1)^{n-\ell} \C\P^{n-3}.$

\begin{rmk}
We already observed that if $\ell=n-2$ then the vectors ${u_i, v_i, \hat{u}_{i}, \hat{v}_{i}}$ form a basis of $T_{[P]}M_r.$ In this case it is straightforward to see that the orientations induced by $A$ and $J$ agree, i.e., $\det(M)=1.$ So the result  sgn$(\det(M)) = (-1)^{n-\ell}$ holds for each $\ell=1, \ldots, n-2.$
\end{rmk}  

\begin{rmk}\label{l}
We assumed in \eqref{ipotesi1} and \eqref{ipotesi2} that the first $\ell$ edges are oriented as the $x$-axis and the following $n-\ell-2$ are conversely oriented. We already pointed out that this assumption is equivalent to choosing a particular class $[P].$ Let us consider another fixed point $[Q]= [\vec{e}]$ of type I. Because the first $n-2$ edges are on the $x$-axis and $\mu=e_1+ \ldots + e_{n-2}=\| \mu \| \i,$ then there exist two subsets $I$ and $I^c$ of $\{ 1, \ldots, n-2 \}$ such that  $I \cap I^c= \emptyset,$ $I \cup I^c =\{ 1,\ldots, n-2\},$ and such that $$ e_i=(r_i, 0,0) \quad \forall i \in I$$ $$  e_i=(-r_i, 0,0) \quad \forall i \in I^c.$$ Let $\ell$ be the cardinality of $I$. If $I=\{ 1, \ldots \ell\}$ then this is the case that we studied in detail. Otherwise, the proof extends word by word just changing $\{1, \ldots, \ell\}$ with $I$ and $\{ \ell +1, \ldots, n-2\}$ with $I^c.$ So a generic point $[Q]$ contributes to the cobordism class of $M_r$ with $(-1)^{n- \ell}\C\P^{n-3}$ where $\ell$ is the number of forward tracks, i.e. $\ell= \sharp \{ e_j \mid e_j \cdot \mu >0 \}.$ 
\end{rmk}

\begin{rmk}
If $n=2m$ then the odd dimensional projective space $\C\P^{n-3}$ is the total space of a circle bundle over the quaternion projective space $\H\P^{m-2},$ and hence is the boundary of the associated disk bundle. So, if $n$ is even $M_r \sim 0.$

\end{rmk}

\begin{rmk}\label{complex}
Since in Section \ref{dim} the complex structure on $M_r$ is analyzed in detail, one might hope that in Theorem \ref{t1} complex cobordism is actually under consideration. Theorem \ref{t1} does not hold when replacing $S^1$-equivariant cobordism with Hamiltonian complex cobordism (cf. \cite{ggk}, Chapter 2, Section 5). In fact Hamiltonian complex cobordant spaces have the same quantization (cf \cite{ggk}) and it is easy to provide two $S^1$-cobordant spaces such that the dimensions of their geometric quantizations are different.
As an example, consider the polygon spaces associated to the length vectors $r_1=(2,3,8,2,4)$ and $r_2=(4,6,16,4,8).$ Both are cobordant to $\C\P^2$ (they are obtained just by rescaling the length vector as in the first example in next Section), but their geometric quantizations are different. In fact, for $M_{r_1}$ the dimension of the space of holomorphic sections of the pre-quantum line bundle is 3, while for $M_{r_2}$ is 28. Details will appear in \cite{bos}. The question is still open about complex cobordism.

\end{rmk}

\section{Some examples}\label{es_cob}

For each length vector $r$ we will analyse which index sets $I$ are $r$-admissible (see Definition \ref{df}). We point out that if $I$ does not satisfy the closing conditions (system \ref{tr}), also its complement $I^c:=\{ 1, \ldots,5\}\setminus I$ does not. Moreover if $I$ is admissible then $I^c $ can't be admissible too, in fact just one between $\sum_{i \in I} \epsilon_i r_i>0$ and $\sum_{i \in I^c} \epsilon_i r_i>0 = - \sum_{i \in I} \epsilon_i r_i>0$ is true.
 In this section we will denote an element of  $(M_r^{S^1})_{isol}$ just by giving the signs of the vectors $e_1, e_2, e_3,$ so for example $++-$ say us $$e_1 = (r_1,0,0), e_2 =(r_2,0,0), e_3 =(-r_3,0,0),$$ and the remaining edges $e_4, e_5$ are determined up to rotations. So the class (uniquely) determined in $M_r$ by $++-$ will be denoted by $P_{++-}.$

In the examples studied the vector of lengths is not normalized (i.e. $\sum_i r_i \ne 2$). This will keep the notation cleaner and is not restrictive because $M_r \simeq M_{\la r}$ for all $\la \in \R^+.$

Each of the following examples is obtained by its previous one by crossing a inner wall in $\Xi,$ or equivalently (because $M_r$ is toric for $n=5$) by chopping off a vertex in the moment polytope $\mu_{T^2}(M_r).$ We will go back to this remark at the end of this section, but this should be kept in mind as looking at the moment polytope.
\begin{enumerate}
\item $ \textbf{ r=(1,1.5,4,1,2)}:$
$M_r$ is a smooth manifold, and the only $r$-admissible set is $I$ $\{ 3\};$ $\ell = | \{ 3\}|= 1$ so the  $\C\P^2$ produced with the surgery around $P_{--+}$ comes with  sign $ (-1)^{5-1} = 1,$ i.e. it comes with the standard orientation.

Other configurations are not $r$-admissible, in fact $\{1,2,3 \},$ $\{1,3 \},$ $\{2,3 \}$ (and their complements) do not satisfy the closing condition (i.e. system \eqref{tr}); $\{ 1,2\}$ is also not $r$-admissible. In fact it is the complement of $\{ 3\}.$
Thus $$M_r \sim \C\P^2. $$
In this case the image $\mu_{{T^2}}(M_r)$ is as in figure \ref{espap1}-(A).

\begin{figure}[htbp]
\psfrag{1}{\footnotesize{$y=x+4$}}
\psfrag{2}{\footnotesize{$y=-x+4$}}
\psfrag{3}{\footnotesize{$y=x-4$}}
\psfrag{a}{\footnotesize{$0.5$}}
\psfrag{a1}{\footnotesize{$1.5$}}
\psfrag{b}{\footnotesize{$2.5$}}
\psfrag{c}{\footnotesize{$1$}}
\psfrag{d}{\footnotesize{$3$}}
\psfrag{x}{\footnotesize{$x$}}
\psfrag{y}{\footnotesize{$y$}}
\psfrag{A}{\footnotesize{(A):$\mu_{{T^2}}(M_r)$ for $r=(1,1.5,4,1,2)$, $M_r \sim \C\P^2.$}}
\psfrag{B}{\footnotesize{(B):$\mu_{{T^2}}(M_r)$ for $r=(0.5,2,4,1,2)$, $M_r \sim 0.$}}
\includegraphics[width=12cm]{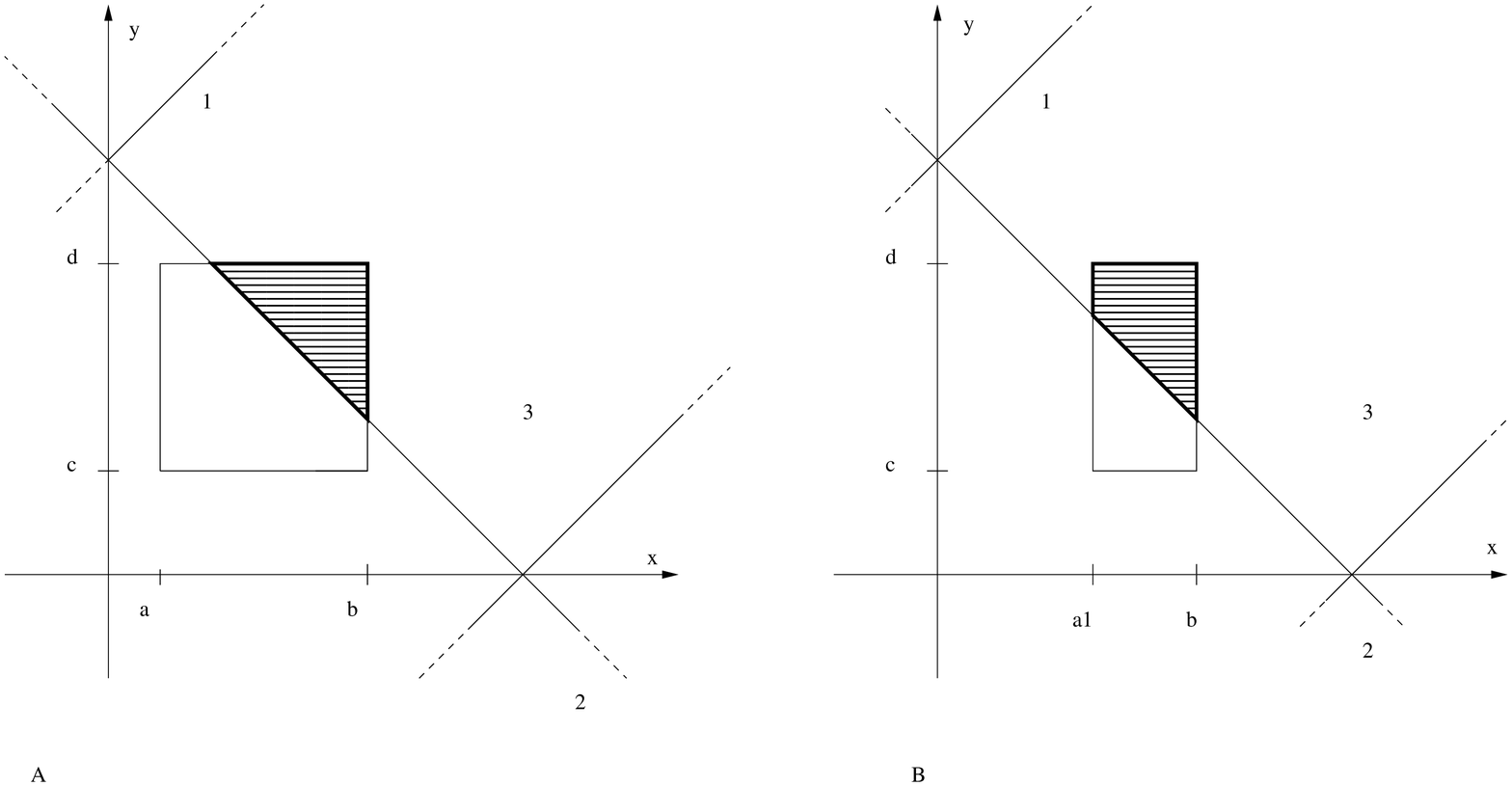}
\caption{Examples 1 and 2}
\label{espap1}
\end{figure}

\item $ \textbf{ r=(0.5,2,4,1,2)}:$
$M_r$ is a smooth manifold, and the $r$-admissible index sets are:\begin{itemize}

\item[$\{2,3\}$]   $\Rightarrow$ $l=2$ $\Rightarrow$ on $T_{P_{-++}}M_r,$ $A=-J$ and $\C\P^2 $ comes with the orientation opposite to the standard one.

\item[$\{3\}$]  $\Rightarrow$ $l=3$  $\Rightarrow$ on $T_{P_{+++}}M_r,$ $A=J$ and $\C\P^2 $ comes with the standard orientation.

\end{itemize}
Thus $$M_r \sim  \C\P^2 \amalg - \C\P^2 \sim 0. $$

For this choice of $r$ the image $\mu_{{T^2}}(M_r)$ is as in Figure \ref{espap1}-(B).

\item $\textbf{ r=(2,0.5,4,0.5,2.5)}$
$M_r$ is a smooth manifold, and the only $r$-admissible set is  $I=\{ 2,3\},$ of cardinality  $\ell=2.$  So on $T_{P_{-++}}M_r,$ $A= (-1)^{n-\ell}= -J$ and $\C\P^2 $ comes with the opposite orientation to the standard one. There are no other $r$-admissible sets. In fact $\{1,2,3 \},$ $\{ 1,2\},$ $\{ 3\},$ $\{ 1,3\}$ and their complements do not satisfy system \eqref{tr}, and neither does $\{ 1\}$ (it is the complement of $\{ 2,3\}$). 

Thus $$M_r \sim  - \C\P^2.$$

\begin{figure}[htbp]
\psfrag{1}{\footnotesize{$y=x+4$}}
\psfrag{2}{\footnotesize{$y=-x+4$}}
\psfrag{3}{\footnotesize{$y=x-4$}}
\psfrag{a}{\footnotesize{$1.5$}}
\psfrag{b}{\footnotesize{$2.5$}}
\psfrag{b1}{\footnotesize{$5.5$}}
\psfrag{c1}{\footnotesize{$1$}}
\psfrag{c}{\footnotesize{$2$}}
\psfrag{d}{\footnotesize{$3$}}
\psfrag{x}{\footnotesize{$x$}}
\psfrag{y}{\footnotesize{$y$}}
\psfrag{A}{\footnotesize{(A):$\mu_{{T^2}}(M_r)$ for $r=(2,0.5,4,0.5,2.5),$}}
\psfrag{C}{\footnotesize{ $M_r \sim - \C\P^2.$}}
\psfrag{B}{\footnotesize{(B):$\mu_{{T^2}}(M_r)$ for $r=(2,3.5,4,1,2),$}}
\psfrag{D}{\footnotesize{ $M_r \sim - 2\C\P^2.$}}
\includegraphics[width=12cm]{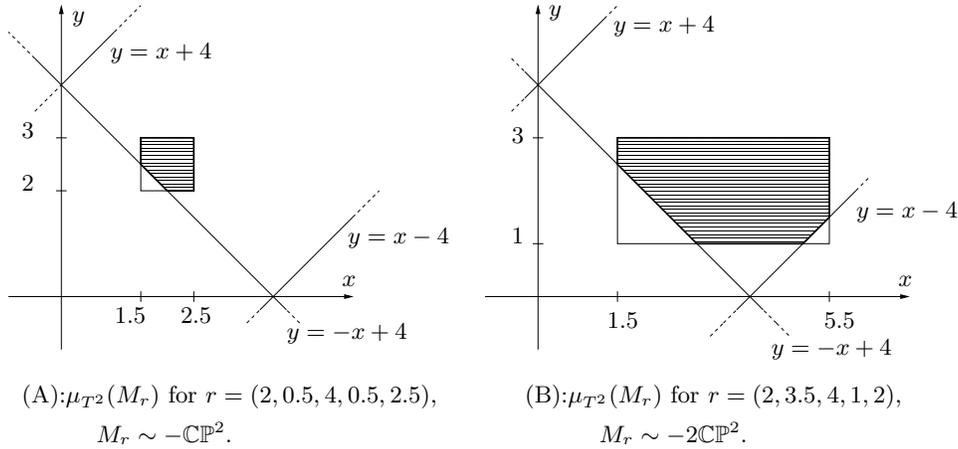}
\caption{Examples 3 and 4.}
\label{espap2}
\end{figure}

The image $\mu_{{T^2}}(M_r)$ of $M_r$ is then the 5-sided polytope in Figure \ref{espap2}-(A).

\item $\textbf{  r=(2,3.5,4,1,2)}$\\
$M_r$ is a smooth manifold, and the $r$-admissible index subsets are $\{ 1,2\}$ and $\{1,3 \}.$ Both of them have cardinality $\ell=2,$ and so they contribute to the cobordism class of $M_r$ with two copies of $- \C\P^2,$ i.e.
$$M_r \sim  - \C\P^2 \amalg - \C\P^2 \sim  -2 \C\P^2. $$

As before, it is immediate to draw the polytope $\mu_{{T^2}}(M_r)$, see Figure \ref{espap2}-(B).

\item $\textbf{ r=(2,3.5,4,3.5,2.5)}$\\
$M_r$ is a smooth manifold, and the $r$-admissible sets are $\{1,2 \},$ $\{1,3 \},$ $\{ 2,3\}.$ All of them have cardinality $\ell =2,$ so the corresponding fixed points contribute to the cobordism class of $M_r$ with a $- \C\P^2.$
Thus $$M_r \sim - \C\P^2 \amalg - \C\P^2 \amalg - \C\P^2 \sim -3 \C\P^2. $$

\begin{figure}[htbp]
\begin{center}
\psfrag{1}{\footnotesize{$y=x+4$}}
\psfrag{2}{\footnotesize{$y=-x+4$}}
\psfrag{3}{\footnotesize{$y=x-4$}}
\psfrag{a}{\footnotesize{$1.5$}}
\psfrag{b}{\footnotesize{$5.5$}}
\psfrag{c}{\footnotesize{$1$}}
\psfrag{d}{\footnotesize{$6$}}
\psfrag{4}{\footnotesize{$4$}}
\psfrag{x}{\footnotesize{$x$}}
\psfrag{y}{\footnotesize{$y$}}
\psfrag{A}{\footnotesize{$\mu_{{T^2}}(M_r)$ for $r=(2,3.5,4,3.5,2.5)$, $M_r \sim -3 \C\P^2$}}
\includegraphics[width=5cm]{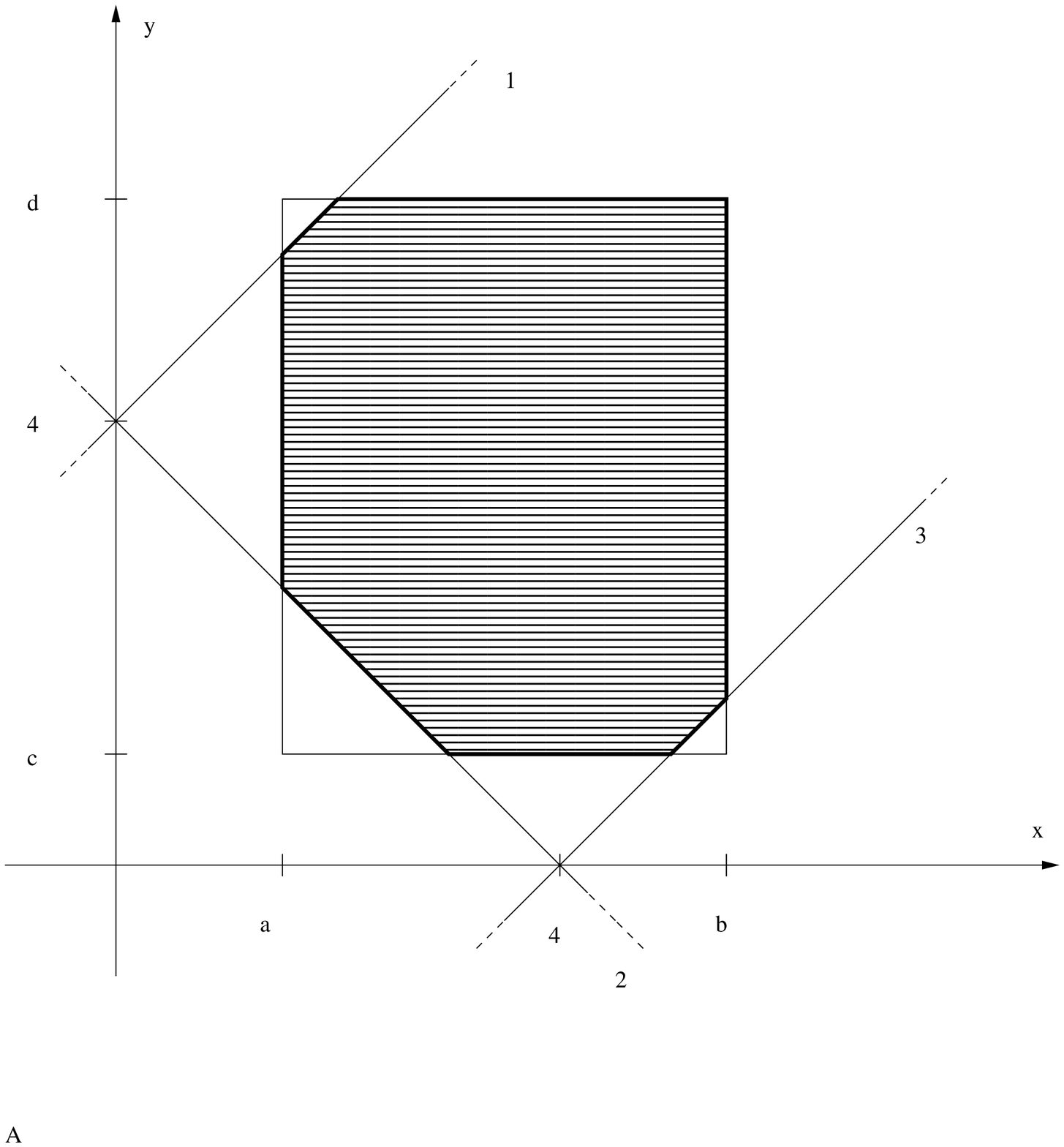}
\end{center}
\caption{Example 3}
\label{figura8}
\end{figure}
\noindent
For this choice of the length vector $r$ the image $\mu_{{T^2}}(M_r)$ is as in Figure \ref{figura8}.

\item $ \textbf{  r=(5,1,4,5,1)}:$
$M_r$ is a smooth manifold. For this choice of $r$ the set $(M_r^{S^1})_{isol}$ is empty. In fact none of the index sets $\{1,2,3 \},$ $\{1,2 \},$  $\{ 1,3\},$  $\{ 2,3\}$ are $r$-admissible, thus $$M_r \sim 0$$ 

\begin{figure}[htbp]
\psfrag{4}{\footnotesize{$4$}}
\psfrag{6}{\footnotesize{$6$}}
\psfrag{0.5}{\footnotesize{$0.5$}}
\psfrag{6.5}{\footnotesize{$6.5$}}
\psfrag{3.5}{\footnotesize{$3.5$}}
\psfrag{2.5}{\footnotesize{$2.5$}}
\psfrag{1'}{\footnotesize{$y=x+4$}}
\psfrag{2'}{\footnotesize{$y=-x+4$}}
\psfrag{3'}{\footnotesize{$y=x-4$}}
\psfrag{1}{\footnotesize{$y=x+3.5$}}
\psfrag{2}{\footnotesize{$y=-x+3.5$}}
\psfrag{3}{\footnotesize{$y=x-3.5$}}
\psfrag{x}{\footnotesize{$x$}}
\psfrag{y}{\footnotesize{$y$}}
\psfrag{A}{\footnotesize{(A):$\mu_{{T^2}}(M_r)$ for $r=(5,1,4,5,1),$ $M_r \sim 0.$}}
\psfrag{B}{\footnotesize{(B):$\mu_{{T^2}}(M_r)$ for $r=(1,1.5,3.5,3,3.5),$}}
\psfrag{C}{\footnotesize{ $M_r \sim 0.$}}
\includegraphics[width=11cm]{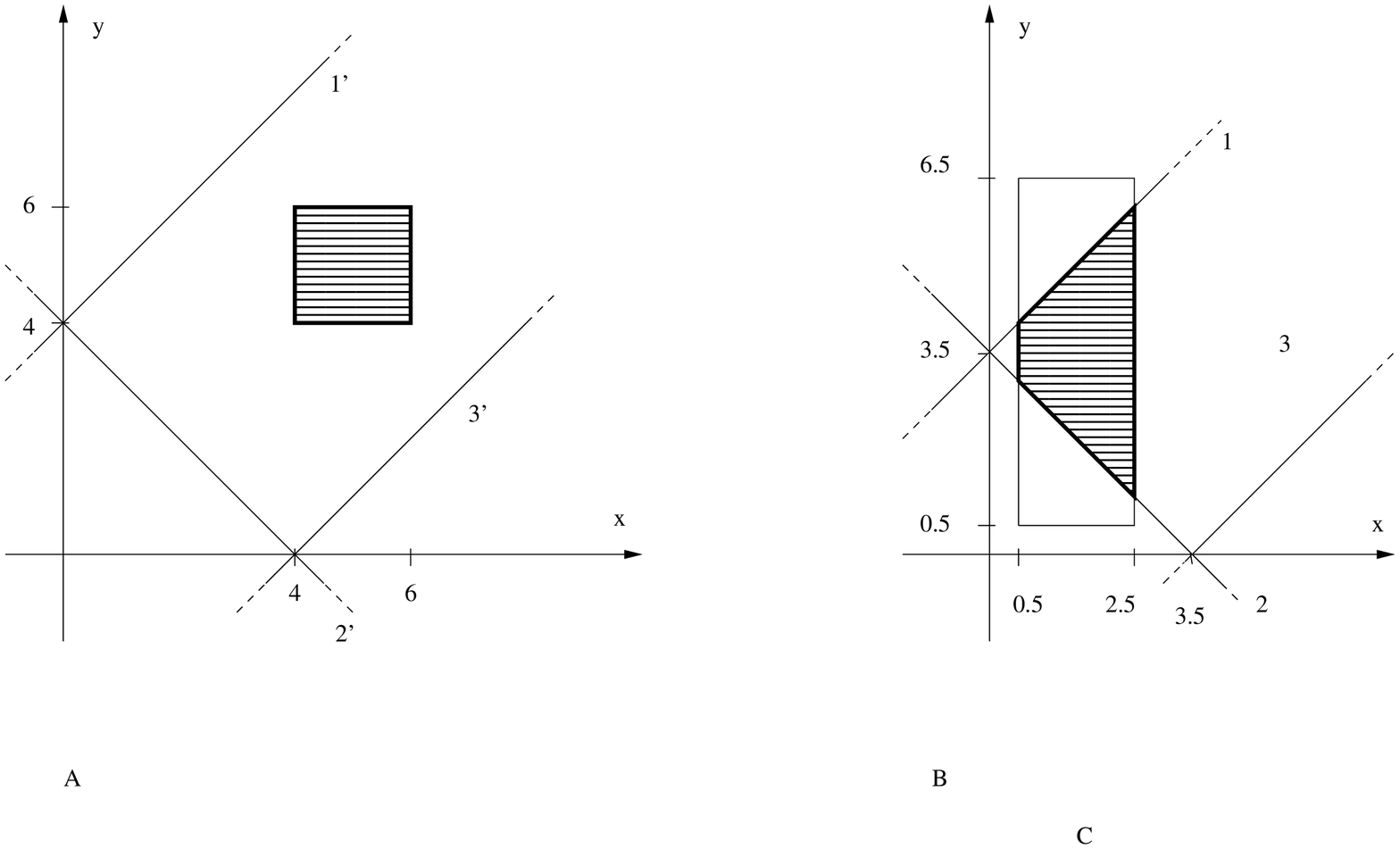}
\caption{Examples 6 and 7.}
\label{espap3}
\end{figure}
and $ \mu_{{T^2}}(M_r)$ is as in Figure \ref{espap3}-(A).

\item $\textbf{ r=(1,1.5,3.5,3,3.5)}:$
$M_r$ is a smooth manifold, and the $r$-admissible index sets are $\{1,2,3 \},$ $\{ 1,3\},$ $\{ 2,3\},$ $\{ 3\}.$ Of these, two have even cardinality and two have odd cardinality, so

$$M_r \sim  \C\P^2 \amalg \C\P^2 \amalg - \C\P^2 \amalg - \C\P^2 \sim 0$$

and the moment image $\mu_{{T^2}}(M_r)$ is as in Figure \ref{espap3}-(B).
\end{enumerate}

Note that the examples above are built by ``chopping off a vertex'' at each step. This has a formal description: ``chopping a vertex'' corresponds to a wall crossing in $\Xi.$ For example the passage from $r$'s such that $\mu_{{T^2}}(M_r)$ is as in Figure \ref{espap1}-(A) to $r$'s such that $\mu_{{T^2}}(M_r)$ is as in Figure \ref{espap1}-(B) corresponds to the crossing of the wall $r_1+r_3= r_2+r_4+r_5.$

This is an expected phenomenon. In fact in the 4-dimensional case ($n=5$) crossing a wall has the effect of blowing up a fixed point (or blowing down, depending on the wall-crossing direction). For this we refer to \cite{io}, where we describe how the diffeotype of $M_r$ changes as $r$ crosses a wall in $\Xi.$

By the notion of admissibility for an index subset $I$, it follows that for $n=5$ these are all the possible cobordism types of $M_r.$ Moreover for $r$'s in the same region of regular values $\Delta \subset \Xi,$ the moment polytope $\mu_{{T^2}}(M_r)$ has the same ``shape'', and its number of edges is an invariant of cobordism.

\begin{rmk}
The manifolds $M_r$ as in Examples 2 and 6-7 have the same cobordism type ($M_r \sim 0$) but different diffeotype, and thus different symplectomorphism type. The moment polytope $\mu_{{T^2}}(M_r)$ contains all the informations
needed to recover the ($T^2$-equivariant) symplectomorphism type (see Delzant \cite{De}, Lerman-Tolman \cite{lt}). For $M_r$'s such that the moment polytope is as in Example 6,and more generally when the opposite edges of the polytope $\mu_{{T^2}}(M_r)$ are parallel, it is well-known that the manifold $M_r$ is diffeomorphic to $\C\P^1 \times \C\P^1$ (see, for example, \cite{cannas}).

Let us now analyze the cases such that the moment polytope has shape as in Figures \ref{espap1}-(B) and \ref{espap3}-(B). Karshon \cite{karshon} finds  explicitly the ($ S^1$-equivariant) symplectomorphism types for these
examples, and, in particular, establishes when they are the same. A possible way to see it is the following: because $\mu_{{T^2}}(M_r)$ is the intersection of the regions $I$ and $\Upsilon,$ its edges are either horizontal, vertical or have slope $\pm1.$ Moreover  there is always a pair of opposite edges which are parallel.  If the normals to the other opposites edges (the non-parallel ones) generate the lattice $\mathbb{Z}^2$ then $M_r$ is diffeomorphic to  $\C\P^2$ blown up at a point; otherwise, if they generate a sublattice of  $\mathbb{Z}^2$ of index two, it is diffeomorphic to  $S^2 \times S^2 \simeq \C\P^1 \times \C\P^1.$ 
This can also be seen by analyzing the graphs associated to the polytopes  as in Figures \ref{espap1}-(B) and \ref{espap3}-(B) (see \cite{Yael}, section 2).

\end{rmk}

\noindent
Center for Mathematical Analysis, Geometry and Dynamical Systems\\
Departamento de Matematica, Instituto Superior Tecnico,\\
1049-001 Lisboa, Portugal.\\
 Fax: (351) 21 8417035\\

Email:amandini@math.ist.utl.pt


\begin{thebibliography}{90}            
\rhead[\fancyplain{}{\bfseries \leftmark}]{\fancyplain{}{\bfseries
\thepage}}
\bibitem[AG]{AG} J.~Agapito, L.~Godinho, \emph{Intersection Numbers of Polygon Spaces}, arXiv:0709.2097, to appear in Trans. Amer. Math. Soc.

\bibitem[BM]{bos} R.~Bos, A.~Mandini, \emph{On the quantization of polygon spaces}, in preparation.

\bibitem[ACL]{cannas} A.~Cannas da Silva, \emph{Symplectic toric manifolds} in ``Symplectic Geometry of Integrable Hamiltonian Systems'' by M.~Audin, A.~Cannas da Silva, E.~Lerman, Birkh\"auser series Advanced Courses in Mathematics - CRM Barcelona, Birkh\"auser, 2003.

\bibitem[De]{De} T.~Delzant, \emph{Hamiltoniens périodiques et images convexes de l'application moment}  Bull. Soc. Math. France  116  (1988),  no. 3, 315--339.

\bibitem[GGK96]{cobordism} V.~Ginzburg, V.~Guillemin, Y.~Karshon, \emph{Cobordism theory and localization formulas for Hamiltonian group actions,} Internat. Math. Res. Notices  1996, no.5,221-234


\bibitem[GGK02]{ggk} V.~Ginzburg, V.~Guillemin, Y.~Karshon, \emph{Moment maps,cobordisms and Hamiltonian group actions,} American Mathematical Society,  2002, no. 98

\bibitem[Ha]{haus} J.C.~Hausmann, \emph{Sur la topologie des bras articulés,}  Algebraic topology Pozna\'n 1989,  146--159, Lecture Notes in Math., 1474, Springer, Berlin, 1991.

\bibitem[HK98]{hk} J.C.~Hausmann, A.~Knutson, \emph{ The cohomology ring of polygon spaces},  Ann. Inst. Fourier (Grenoble)  48  (1998),  no. 1, 281--321.

\bibitem[HK97]{hk1} J.C.~Hausmann, A.~Knutson, \emph{ Polygon spaces and Grassmannians},  Enseign. Math. (2)  43  (1997),  no. 1-2, 173--198.

\bibitem[K]{kamiyama} Y.~Kamiyama, \emph{ Chern numbers of the moduli space of spatial polygons,}  Kodai Math. J.  23  (2000),  no. 3, 380--390.

\bibitem[KM]{km} M.~Kapovich, J.J.~Millson, \emph{The symplectic geometry of polygons in Euclidean space,}  J. Differential Geom.  44  (1996),  no. 3, 479--513.


\bibitem[Ka]{karshon} Y.~Karshon, \emph{ Maximal tori in the symplectomorphism groups of Hirzebruch surfaces},  Math. Res. Lett.  10  (2003),  no. 1, 125--132

\bibitem[Ka99]{Yael} Y.~Karshon, \emph{ Periodic Hamiltonian flows on four-dimensional manifolds},  Mem. Amer. Math. Soc.  141  (1999),  no. 672, viii+71 pp.

\bibitem[LT]{lt} E.~Lerman, S.~Tolman, \emph{Hamiltonian torus actions on symplectic orbifolds and toric varieties}, Trans. Amer. Math. Soc.  349  (1997),  no. 10, 4201--4230.


\bibitem[Ma]{Martin} S.K.~Martin, \emph{ Transversality theory, cobordisms, and invariants of symplectic quotients},  arXiv:math/0001001.

\bibitem[M]{io} A.~Mandini, \emph{The Duistermaat--Heckman formula and the cohomology of moduli spaces of polygons}, in preparation.   

\end{thebibliography}
\end{document}